\documentclass[11pt]{article}
\usepackage{amsfonts}
\usepackage{mathrsfs}
\usepackage{amsthm}
\usepackage{amsmath}
\usepackage{amssymb}
\usepackage{latexsym}
\usepackage{graphicx}
\usepackage{color,varioref}
\usepackage{longtable}
\usepackage{cite}
\usepackage{enumerate}
\usepackage{appendix}
\usepackage{url}
\definecolor{red}{rgb}{0.9,0,0}

\definecolor{green}{rgb}{0,0.9,0}

\definecolor{blue}{rgb}{0,0,0.9}

\def\Limsup{\mathop{{\rm Lim}\,{\rm sup}}}

\def\disp{\displaystyle}

\def\ox{\overline{x}}

\def\disp{\displaystyle}
\def\tto{\rightrightarrows}

\def\ve{\varepsilon}
\def\epsilon{\varepsilon}
\def\ox{\bar{x}}

\def\ou{\bar{u}}

\def\gph{\mbox{\rm gph}\,}

\def\dom{\mbox{\rm dom}\,}

\def\O{\Omega}

\def\emp{\emptyset}

\def\dd{\delta}

\def\th{\theta}

\def\Limsup{\mathop{{\rm Lim}\,{\rm sup}}}

\def\Limsup{\mathop{{\rm Lim}\,{\rm sup}}}

\newtheorem{theorem}{Theorem}[section]
\newtheorem{proposition}{Proposition}[section]
\newtheorem{lemma}{Lemma}[section]
\newtheorem{corollary}{Corollary}[section]

\newtheorem{definition}{Definition}[section]
\newtheorem{assumption}{Assumption}[section]

\begin{document}

\title{\bf Second-Order Characterizations of Tilt Stability\\ in Composite Optimization}

\author{Boris S. Mordukhovich\footnote{Department of Mathematics, Wayne State University, Detroit, MI 48202 (aa1086@wayne.edu).} \quad Peipei Tang\footnote{School of Computer and Computing Science, Hangzhou City University, Hangzhou, China (tangpp@hzcu.edu.cn).}\quad Chengjing Wang\footnote{School of Mathematics, Southwest Jiaotong University, Chengdu, China (renascencewang@hotmail.com).}}

\maketitle
\vspace*{-0.2in}
\begin{abstract}
Tilt stability is a fundamental concept of variational analysis and optimization that plays a pivotal role in both theoretical issues and numerical computations. This paper investigates tilt stability of local minimizers for a general class of composite optimization problems in finite dimensions, where extended-real-valued objectives are compositions of parabolically regular and smooth functions. Under the weakest metric subregularity constraint qualification and other verifiable conditions, we establish unified neighborhood and pointbased characterizations of tilt stability via second-order generalized differentiation. The obtained results provide a rigorous theoretical foundation for further developments on variational stability and numerical algorithms of optimization and related topics.
\end{abstract}\vspace*{-0.1in}
\begin{keywords}
variational analysis, generalized differentiation, composite optimization, tilt stability, second-order characterizations, parabolic regularity
\end{keywords}\vspace*{0.1in}
{\bf Mathematics Subject Classification (2020)} 49J52, 49J53, 90C31
\vspace*{-0.12in}

\section{Introduction}\label{intro}\vspace*{-0.05in}

Variational analysis has unequivocally demonstrated that well-formulated stability concepts constitute fundamental pillars of optimization theory and its practical applications. Particularly important are Lipschitzian stability properties, which provide both qualitative characterization and quantitative measures that have been recognized to be indispensable for rigorous analysis and justification of numerical algorithms in problems of optimization and operations research. 

This paper is devoted to studying the notion of {\em tilt stability} for local minimizers, which was introduced by Poliquin and Rockafellar \cite{Poliquin1998} for optimization problems written in the unconstrained extended-real-valued format as follows. Given a function $f:\mathcal{R}^{m}\rightarrow\overline{\mathcal{R}}:=\mathcal{R}\cup\{\infty\}$ and a point $x\in\operatorname{dom}f:=\{
x\in\mathcal{R}^m\;|\;f(x)<\infty\}$, it is said that  $x$ is a {\em tilt-stable local minimizer} of $f$ with modulus $\kappa>0$ if there exists $\delta>0$ such that the mapping 
\begin{align*}
\mathcal{M}_{\delta}:u\mapsto\mathop{\operatorname{arg}\min}_{\|x'-x\|\leq\delta}\{f(x')-f(x)-\langle u,x'-x\rangle\}
\end{align*}
is single-valued and Lipschitz continuous with modulus $\kappa$ on some neighborhood of $u=0$ and that $\mathcal{M}_{\delta}(0)=\{x\}$. The fundamental result of \cite{Poliquin1998} shows that, for a broad class of prox-regular and subdifferentially continuous functions, the tilt stability of local minimizers is completely characterized by the positive-definitiness of the {\em second-order subdifferential/generalized Hessian} in the sense of Mordukhovich \cite{Mordukhovich1992}. Over the years, a variety of second-order characterizations of tilt-stable minimizers for different classes of constrained optimization problems have been established and applied to qualitative and numerical aspects of optimization, control, and operations research. We refer the reader to, e.g., \cite{BenkoGfrererMordukhovich2019,Bonnans2000,Chieu2018,Drusvyatskiy2013,Drusvyatskiy2014,GfrererMordukhovich2015,Khanh2023,MordukhovichNghia2015,
MordukhovichOutrataRamirez2015,MordukhovichNghiaRockafellar2015,Mordukhovich2021} and the bibliographies therein for more details. The recent book \cite{Mordukhovich2024} contains an extensive account of the major achievements on tilt stability, related properties, and their applications in finite and infinite dimensions. Quite new characterizations of tilt-stabile local minimizers for unconstrained extended-real-valued functions with prox-regularity but without subdifferential continuity requirements have been derived in the fresh papers \cite{Helmut2025,Khanh2025,Rockafellar2025} via quadratic bundles.

In this paper, we address tilt stability in composite optimization problem of the type
\begin{align}\label{P}
\mbox{minimize }\; f_{0}(x)+g(F(x))\;\mbox{ over }\;x\in\mathcal{R}^n,
\end{align}
where $f_{0}:\mathcal{R}^{n}\rightarrow\mathcal{R}$, $F:\mathcal{R}^{n}\rightarrow\mathcal{R}^{m}$ are smooth mappings, and where $g:\mathcal{R}^{m}\rightarrow \overline{\mathcal{R}}$ is a lower semicontinuous (l.s.c.) convex function. The framework of \eqref{P} covers many particular classes of problems important in optimization theory and applications, e.g., nonlinear, second-order cone, and semidefinite programmings, where $g$ is the indicator functions of the positive orthant, the second-order (Lorentz, ice-cream) cone, and the positive semidefinite cone, respectively. 

The first characterization of tilt-stable local minimizers was obtained in \cite{Mordukhovich2012} for problems of nonlinear programming (NLPs) in terms of Robinson's strong second-order sufficient condition under the classical linear independence constraint qualification (LICQ). The results of \cite{Mordukhovich2012} were given in the {\em pointbased form}, i.e., they were expressed entirely at the point in question, which is the most important for applications. Pointbased extensions of this characterization and related results for nonpolyhedral problems of conic programming (including second-order cone programs and semidefinite programs) were later established in \cite{MordukhovichNghiaRockafellar2015,MordukhovichOutrataRamirez2015,MordukhovichOutrataSarabi2014} under certain {\em nondegeneracy conditions}, which are conic counterparts of LICQ. 

More delicate characterizations of tilt stability for NLPs were obtained in \cite{MordukhovichNghia2015} without imposing LICQ (and thus the uniqueness of Lagrange multipliers), but in a {\em neighborhood form} by invoking points near the reference one. A broad spectrum of pointbased characterizations of (as well as separate sufficient and necessary conditions for) tilt-stable minimizers in NLPs, involving only the point in question, were established in \cite{GfrererMordukhovich2015} without any nondegeneracy assumptions while under the novel {\em metric subregularity constraint qualification} (MSCQ). The latter qualification condition was further investigated and applied in many publications; see, e.g., \cite{Chieu2018,Chieu2021,Helmut2019,Gfrerer2022,Khanh2023,Mohammadi2021}. 

To the best of our knowledge, the only paper providing pointbased characterizations of tilt-stable minimizers in nonpolyhedral optimization problems without nondegeneracy is  \cite{BenkoGfrererMordukhovich2019}, where rather complicated conditions of this type are obtained for second-order cone programming (SOCP) with a single Lorentz cone under the SOCP counterpart of MSCQ. Up to now, the understanding of tilt stability for local minimizers in nonpolyhedral problems, particularly for structural composite problems of type \eqref{P}, remains limited in the absence of nondegeneracy.\vspace*{0.03in}

The main goal of this paper is to derive pointwise {\em no-gap second-order necessary and sufficient conditions} for tilt stability of local minimizers in a general class of nonpolyhedral problems of composite optimization \eqref{P} under MSCQ and other mild assumptions including {\em parabolic regularity} of $g$. To this end, we introduce a novel {\em second-order variational function} defined via the given data of \eqref{P}. The established generalized differential properties of this function are of their own interest  while eventually leading us to the desired pointbased characterizations of tilt stability in \eqref{P}. Involving advanced matrix analysis and calculations allows us constructively implement the most challenging general assumption in the case of {\em spectral norm function}, which is particularly important for applications to machine learning and data science.

The remainder of this paper is structured as follows. In Section~\ref{sec:prel}, we  
overview some basic constructions of variational analysis and generalized differentiation broadly used in the paper and present necessary preliminaries. Section~\ref{sec:2nd-order} introduces the second-order variational function and develops its generalized differentiable properties. Section~\ref{sec:no-gap} establishes the main results on neighborhood and pointbased characterizations of tilt stability. After the concluding Section~\ref{conc}, which summarizes the major achievements of the paper and discusses further perspectives, we present the technical Appendix~\ref{appendices} containing a detailed verification of the most involved assumption for the highly important case of the spectral norm function.\vspace*{0.05in}

The notation of the paper is standard in variational analysis; see, e.g., \cite{Mordukhovich2018,Rockafellar1998}. For the reader's convenience, recall that 
the closed unit ball in $\mathcal{R}^{m}$ is denoted by $\mathbb{B}_{\mathcal{R}^{m}}$, while $\mathbb{B}(x,\varepsilon):=\{u\in\mathcal{R}^{m}\, |\, \|u-x\|\leq\varepsilon\}$ for fixed $x\in\mathcal{R}^{m}$ and $\varepsilon>0$. The indicator function of $C$ is defined by $\delta_{C}(x):=0$ if $x\in C$ and $\delta_{C}(x):=\infty$ otherwise. Given $x\in\mathcal{R}^{m}$, the symbols $\operatorname{dist}(x,C)$ and $\operatorname{\Pi}_{C}(x)$ signify the distance from $x$ to $C$ and the projection of $x$ onto $C$, respectively. For a sequence $\{x^{k}\}$, the symbol  $x^{k}{\xrightarrow{C}}x$ means that $x^{k}\rightarrow x$ with $x^{k}\in C$ as $k\in\mathbb N:=\{1,2,\ldots\}$. For a nonempty closed convex cone $C$, the polar cone of $C$ is defined by $C^{\circ}:=\{y\,|\, \langle x,y\rangle\leq 0\;\mbox{ as }\;x\in C\}$,  and the affine space of $C$ is $\operatorname{aff}C:=C-C$, which is the smallest subspace containing $C$. The $n\times n$ identity matrix is denoted by $I_{n}$ (with omitting the subscript when the dimension is evident), and $\operatorname{Diag}(x)$ denotes the diagonal matrix with the vector $x$ standing in the diagonal. The symbol $A^{\dagger}$ indicates the Moore-Penrose inverse for a given matrix $A\in\mathcal{R}^{m\times n}$. We write $\Omega_1\subset\Omega_2$ to  indicate that the set $\Omega_1$ is smaller than or equal to $\Omega_2$.  Given a set-valued mapping multifunction $S\colon\mathcal{R}^{n}\tto\mathcal{R}^{m}$, the (Painlev\'{e}-Kuratowski) {\em outer limit} of $S$ at $\ox$ is 
\begin{equation}\label{pk}
\Limsup_{x\to\ox}S(x):=
\big\{y\in\mathcal{R}^{m}\;\big|\;\exists\,x_k\to\ox,\;y_k\to y\;\mbox{ with }\;y_k\in S(x_k)\big\},
\end{equation}
while the standard upper limit of scalars is denoted by `$\limsup$'. 
\vspace*{-0.12in}

\section{Basic Definitions and Preliminaries}\label{sec:prel}\vspace*{-0.05in}

In this section, we overview fundamental concepts of variational analysis and generalized differentiation that are broadly used in what follows; see the books \cite{Mordukhovich2018,Mordukhovich2024,Rockafellar1998} and further references below for more details. Basic assumptions on the data of \eqref{P} are also formulated below.

Given a set $C\subset\mathcal{R}^{m}$, the (Bouligand-Severi) {\em tangent cone} to $C$ at $x\in C$ is defined by
\begin{align}\label{tan}
\mathcal{T}_{C}(x):=\Limsup_{t\downarrow 0}\frac{C-x}{t}=\left\{w\in\mathcal{R}^{m}\,\left|\,\exists\, t^{k}\downarrow 0,\;x^{k}{\xrightarrow{C}} x,\; \frac{x^{k}-x}{t^{k}}\rightarrow w \right.\right\}
\end{align}
in terms of the outer limit \eqref{pk} of $G(t):=t^{-1}(C-x)$ with $t\in\mathcal{R}$. The (Fr\'{e}chet) {\em regular normal cone} to $C$ at $x$ can be defined equivalently by the equalities
\begin{align}\label{rnc}
\widehat{\mathcal{N}}_{C}(x):=\left\{y\in\mathcal{R}^{m}\, \left|\, \limsup_{{x'\xrightarrow{C} x}\atop{x'\neq x}}\frac{\langle y,x'-x\rangle}{\|x'-x\|}\leq 0\right.\right\}=\mathcal{T}_{C}^{\circ}(x),
\end{align}
while the (Mordukhovich) {\em limiting normal cone} to $C$ at $x$ is given via \eqref{pk} as
\begin{align}\label{lnc}
\mathcal{N}_{C}(x):=\Limsup_{x'\xrightarrow{C} x}\widehat{\mathcal{N}}_{C}(x').
\end{align}
There are several equivalent representations of \eqref{lnc} that can be found in the aforementioned references. If $C$ is convex, both regular and the limiting normal cones reduce to the normal cone of convex analysis, while \eqref{lnc} is often {\em nonconvex} 
even for simple sets on the plane; e.g., when $C$ is the graph of the simplest nonsmooth convex function $f(x):=|x|$ at $0\in\mathcal{R}^2$. Nevertheless, the limiting normal cone \eqref{lnc} and the associated generalized differentiation constructions for functions and multifunctions enjoy {\em full calculus}, which is much better than for their convex counterparts as, e.g., \eqref{rnc}. This is based on {\em extremal/variational principles} of variational analysis.

The tangent and normal cones to graphs of set-valued (and hence single-valued) mappings generate the corresponding derivative-like constructions for them, which play a highly important role in variational analysis. For a set-valued mapping $S:\mathcal{R}^{n}\rightrightarrows\mathcal{R}^{m}$, consider its domain $\dom S:=\{x\in\mathcal{R}^n\,|\,S(x)\ne\emp\}$, graph $\gph S:=\{(x,y)\in\mathcal{R}^{n+m}\,|\,y\in S(x)\}$, and range $\operatorname{rge}S:=\{u\in\mathcal{R}^{m}\,|\,\exists\,x\in\mathcal{R}^{n}\, \mbox{with}\, u\in S(x)\}$. The inverse mapping $S^{-1}:\mathcal{R}^{m}\rightrightarrows\mathcal{R}^{n}$ is defined by $S^{-1}(u):=\{x\in\mathcal{R}^{n}\,|\, u\in S(x)\}$. Recall next the constructions of the graphical derivative and coderivatives of multifunctions generated by \eqref{tan}, \eqref{rnc}, and \eqref{lnc}, respectively. 

\begin{definition}\label{der-cod} Let $S:\mathcal{R}^{n}\rightrightarrows\mathcal{R}^{m}$ with $x\in\operatorname{dom}S$. Then the {\sc graphical derivative} of $S$ at $x$ for $u\in S(x)$ is the mapping $DS(x,u):\mathcal{R}^{n}\rightrightarrows\mathcal{R}^{m}$ defined by
\begin{align}\label{gra-der}
DS(x,u)(w):=\left\{z\in\mathcal{R}^{m}\,\left|\,(w,z)\in\mathcal{T}_{\operatorname{gph}S}(x,u)\right.\right\},\quad w\in\mathcal{R}^{n}.
\end{align}
The {\sc regular coderivative} $\widehat{D}^{*}S(x,u):\mathcal{R}^{m}\rightrightarrows\mathcal{R}^{n}$ of $S$ at $x$ for $u\in S(x)$ is defined by
\begin{align}\label{reg-cod}
\widehat{D}^{*}S(x,u)(w):=\left\{
z\in\mathcal{R}^{n}\,\left|\,(z,-w)\in\widehat{\mathcal{N}}_{\operatorname{gph}S}(x,u)\right.\right\},\quad w\in\mathcal{R}^{m}.
\end{align}
The {\sc limiting coderivative} $D^{*}S(x,u):\mathcal{R}^{m}\rightrightarrows\mathcal{R}^{n}$ of $S$ at $x$ for $u\in S(x)$ is the mapping 
\begin{align}\label{lim-cod}
D^{*}S(x,u)(w):=\left\{z\in\mathcal{R}^{n}\,\left|\,(z,-w)\in\mathcal{N}_{\operatorname{gph}S}(x,u)\right.\right\},\quad w\in\mathcal{R}^{m}.
\end{align}
When $S$ is single-valued at $x$, we skip $u=S(x)$ in the notations \eqref{gra-der}--\eqref{lim-cod}.
\end{definition}

If $S$ is single-valued and Fr\'echet differentiable at $x$ with the Jacobian matrix $\nabla S(x)$, then we get $DS(x,u)(w)=\{\nabla S(x)w\}$ and $\widehat D^*S(x,u)(w)=\{\nabla S(x)^Tw\}$, where $A^T$ indicate the matrix transposition. If $S$ is strictly differentiable at $x$ (in particular, ${\cal C}^1$-smooth around this point), then $D^{*}S(x,u)(w)=\{\nabla S(x)^Tw\}$. In  general, all the mappings in \eqref{gra-der}--\eqref{lim-cod} are positively homogeneous in $w$. It is easy to check that relationships
\begin{align}\label{eq:inverse-graphical-derivative}
\begin{aligned}
z\in DS(x,u)(w)\,&\Longleftrightarrow\, w\in D(S^{-1})(u,x)(z),\\ 
z\in D^{*}S(x,u)(w)\,&\Longleftrightarrow\, -w\in D^{*}(S^{-1})(u,x)(-z),
\end{aligned}
\end{align}
and similarly for the regular coderivative of the inverse. For sums of set-valued mappings $S:\mathcal{R}^{n}\rightrightarrows\mathcal{R}^{m}$ and ${\cal C}^1$-smooth mappings $h:\mathcal{R}^{n}\rightarrow\mathcal{R}^{m}$, we get
\begin{align*}
\begin{aligned}
\mathcal{T}_{\operatorname{gph}(h+S)}(x,u)&=\big\{(w,\nabla h(x)w+z)\,\big|\, (w,z)\in\mathcal{T}_{\operatorname{gph} S}(x,u-h(x))\big\},\\
\widehat{\mathcal{N}}_{\operatorname{gph}(h+S)}(x,u)&=\big\{(w-\nabla h(x)^{T}z,z)\,\big|\,(w,z)\in\widehat{\mathcal{N}}_{\operatorname{gph}S}(x,u-h(x))\big\},
\end{aligned}
\end{align*}
which readily yields the equality sum rules
\begin{align}\label{eq:coderivative-sum}
\begin{aligned}
D(h+S)(x,u)(w)&=\nabla h(x)w+DS(x,u-h(x))(w),\\
D^{*}(h+S)(x,u)(z)&=\nabla h(x)^{T}z+D^{*}S(x,u-h(x))(z).
\end{aligned}
\end{align}

When $S\colon\O\to\mathcal{R}^m$ is single-valued and Lipschitz continuous on an open set $\O\subset\mathcal{R}^n$, yet another generalized derivative is used in what follows. By the classical Rademacher's theorem, such a mapping $S$ is differentiable almost everywhere in $\Omega$. Denoting by $\Omega_{S}\subset\O$ be the set of points where $S$ is differentiable, consider the collection of matrices
\begin{align}\label{B-sub}
\partial_{B}S(x):=\left\{V\in\mathcal{R}^{m\times n}\,\left|\,\exists\,  x^{k}\xrightarrow{\Omega_{S}} x,\,\nabla S(x^{k})\rightarrow V\;\mbox{ as }\;k\to\infty\right.\right\},
\end{align}
which is a nonempty compact subset of $\mathcal{R}^{m\times n}$ known as the {\em B-subdifferential} of $S$ at $x$. The convex hull of \eqref{B-sub} is the (Clarke) {\em generalized Jacobian} of $S$ at $x$ denoted by ${\cal J}S(x)$.

To proceed further, recall that a multifunction  $S:\mathcal{R}^{n}\rightrightarrows\mathcal{R}^{m}$ is {\em calm} at $\bar{z}$ for $\bar{w}$ if there exist constants $\kappa\ge 0$ and $\varepsilon,\delta>0$ such that
\begin{equation}\label{calm}
S(z)\cap\mathbb{B}(\bar{w},\delta)\subset S(\bar{z})+\kappa\|z-\bar{z}\|\mathbb{B}_{\mathcal{R}^{m}}\;\mbox{ for all }\;z\in\mathbb{B}(\bar{z},\varepsilon).
\end{equation}
It is said that $S:\mathcal{R}^{n}\rightrightarrows\mathcal{R}^{m}$ is {\em metrically subregular} at $\bar{z}$ for $\bar{w}$ if
\begin{equation}\label{subreg}
\operatorname{dist}(z,S^{-1}(\bar{w}))\leq\kappa\operatorname{dist}(\bar{w},S(z)\cap\mathbb{B}(\bar{w},\delta))
\end{equation}
with the same choice of $z,\kappa,\ve,\dd$  as in \eqref{calm}. As well-recognized in variational analysis, the calmness and metric subregularity properties of mappings are equivalent up to the passing of their inverses. Based on \eqref{subreg}, we now define the aforementioned qualification condition for general constraint systems that is employed below for problem \eqref{P}. 

\begin{definition}\label{def:mscq} Let $G:\mathcal{R}^{n}\rightarrow\mathcal{R}^{m}$ be a single-valued mapping, and let $C$ be a subset of $\mathcal{R}^{m}$. The constraint system $G(x)\in C$ satisfies the the {\sc metric subregularity constraint qualification} $($MSCQ$)$ at a feasible point $x$ if the set-valued mapping $M(x):=G(x)-C$ is metrically subregular at $x$ for the origin, i.e., there exists a modulus $\kappa>0$ along with $\varepsilon>0$ such that for all $x'\in\mathbb{B}(x,\varepsilon)$ we have the estimate
\begin{align}\label{eq:MSCQ}
\operatorname{dist}(x',M^{-1}(0))\le\kappa\,\operatorname{dist}(0,M(x')).
\end{align}
\end{definition}
It is worth emphasizing that the MSCQ from Definition~\ref{def:mscq} is significantly weaker than the standard {\em metric regularity} of  $M(\cdot)$, which corresponds to \eqref{eq:MSCQ} with replacing $\bar{z}=0$ in both parts therein by all $z\in\mathcal{R}^{m}$ with small norm.\vspace*{0.05in}

Next we pass to extended-real-valued functions. Given $f:\mathcal{R}^{m}\rightarrow\overline{\mathcal{R}}$,  consider its domain and epigraph defined, respectively, by
\begin{equation*}
\operatorname{dom}f:=\{x\in\mathcal{R}^{m}\,|\,f(x)<\infty\},\quad \operatorname{epi}f:=\{(x,y)\in\mathcal{R}^{m}\times\mathcal{R}\,|\,x\in\operatorname{dom}f,\,f(x)\leq y\}
\end{equation*}
and associate with $f$ the {\em Fenchel conjugate} 
\begin{align*}
f^{*}(x)&:=\sup_{u\in\operatorname{dom}f}\left\{\langle x,u\rangle-f(u)\right\}.
\end{align*}
We say $f$ is proper if $f(x)>-\infty$ for all $x\in\mathcal{R}^{m}$ and $\operatorname{dom}f\ne\emp$. Fixed $\sigma>0$, the {\em Moreau envelope} $e_{\sigma f}$ and the {\em proximal mapping} $\operatorname{Prox}_{\sigma f}$ corresponding to a proper l.s.c.\ function $f$ are defined, respectively, by the formulas
\begin{align}\label{moreau}
e_{\sigma f}(x)&:=\inf_{u\in\mathcal{R}^{m}}\left\{f(u)+\frac{1}{2\sigma}\|u-x\|^{2}\right\},\\ \label{prox}
\operatorname{Prox}_{\sigma f}(x)&:=\mathop{\operatorname{arg}\min}\limits_{u\in\mathcal{R}^{m}}\left\{f(u)+\frac{1}{2\sigma}\|u-x\|^{2}\right\}.
\end{align}
If $f$ is l.s.c.\ and convex, it is well known from convex analysis \cite{Rockafellar1970} that \eqref{prox} is single-valued and Lipschitz continuous, and that \eqref{moreau} is convex and ${\cal C}^1$-smooth with the gradient
\begin{align*}
\nabla e_{\sigma f}(x)=\frac{1}{\sigma}\big(x-\operatorname{Prox}_{\sigma f}(x)\big)=\operatorname{Prox}_{\sigma^{-1}f^{*}}(\sigma^{-1}x).
\end{align*}
Addressing further subdifferentials of extended-real-valued functions, define them geometrically via the normal cones in \eqref{rnc} and \eqref{lnc}, while observing that they admit various analytic representations; see \cite{Mordukhovich2018,Rockafellar1998}. In this way, the {\em regular subdifferential} and the {\em limiting subdifferential} of $f\colon\mathcal{R}^m\to\overline{\mathcal{R}}$ at $x\in\dom f$ are given, respectively, by
\begin{align*}
\widehat{\partial}f(x)&:=\big\{v\in\mathcal{R}^{m}\,\big|\,(v,-1)\in\widehat{\mathcal{N}}_{\operatorname{epi}f}(x,f(x))\big\},\\ %\label{lim-sub}
\partial f(x)&:=\big\{v\in\mathcal{R}^{m}\,\big|\,(v,-1)\in\mathcal{N}_{\operatorname{epi}f}(x,f(x))\big\}.
\end{align*}
The {\em singular subdifferential} of $f$ at $x$ is defined by \begin{align*}
\partial^{\infty}f(x):=\left\{v\in\mathcal{R}^{m}\,\left|\,\exists \,t_{k}\downarrow0,x^{k}\rightarrow x\;\mbox{with}\; f(x^{k})\rightarrow f(x),\;v^{k}\in\widehat{\partial}f(x^{k}),\;t_{k}v^{k}\rightarrow v\right.\right\}.
\end{align*}
An l.s.c.\ function $f$ is {\em prox-regular} at $\bar{x}\in\operatorname{dom}f$ for $\bar{u}\in\partial f(\bar{x})$ if there are $\varepsilon>0$ and $\rho\geq0$ with
\begin{align*}
f(x')\geq f(x)+\langle u,x'-x\rangle-\frac{\rho}{2}\|x'-x\|^{2}
\end{align*}
for all $x'\in\mathbb{B}(\bar{x},\varepsilon)$ such that $x\in\mathbb{B}(\bar{x},\varepsilon)$, $u\in\partial f(x)\cap\mathbb{B}(\bar{u},\varepsilon)$, and $f(x)<f(\bar{x})+\varepsilon$. If this holds for all $\bar{u}\in\partial f(\bar{x})$, $f$ is said to be prox-regular at $\bar{x}$. We say that $f$ is {\em subdifferentially continuous} at $\bar{x}$ for $\bar{u}\in\partial f(\bar{x})$ if $f(x^{k})\rightarrow f(\bar{x})$ whenever $(x^{k},u^{k})\rightarrow(\bar{x},\bar{u})$ with $u^{k}\in\partial f(x^{k})$. If this holds for all $\bar{u}\in\partial f(\bar{x})$, $f$ is called subdifferentially continuous at $\bar{x}$. Any l.s.c.\ convex function is prox-regular and subdifferentially continuous at every point in its domain.

Consider now generalized directional derivatives of extended-real-valued functions. Given $f\colon\mathcal{R}^m\to\overline{\mathcal{R}}$ and $x\in\operatorname{dom}f$, the (first-order) {\em subderivative} of $f$ at $x$ is defined by
\begin{align}\label{subder}
df(x)(d):=\liminf_{{t\downarrow 0}\atop{d'\rightarrow d}}\frac{f(x+td')-f(x)}{t},\quad d\in\mathcal{R}^m,
\end{align}
where we can put $d'=d$ if $f$ is Lipschitz continuous around $x$. To define the second-order counterpart of \eqref{subder}, form the second-order difference quotient of $f$ at $x\in\operatorname{dom}f$ for some $u$ by
\begin{align*}
\Delta_{\tau}^{2}f(x,u)(d):=\frac{f(x+\tau d)-f(x)-\tau\langle u,d\rangle}{\frac{1}{2}\tau^{2}}\ \ \mbox{ with }\ d\in\mathcal{R}^{m},\ \tau>0
\end{align*}
and say that $d^{2}f(x,u):\mathcal{R}^{m}\rightarrow\overline{\mathcal{R}}$ is the {\em second subderivative} of $f$ at $x$ for $u$ if
\begin{align}\label{2subder}
d^{2}f(x,u)(d):=\liminf_{{\tau\downarrow 0}\atop{d'\rightarrow d}}\Delta_{\tau}^{2}f(x,u)(d'),\quad d\in\mathcal{R}^{m}.
\end{align}
The function $f$ is called {\em twice epi-differentiable} at $x$ for $u$ if $\Delta_{\tau}^{2}f(x,u)$ epi-converges to $d^{2}f(x,u)$ as $\tau\downarrow 0$. To proceed further, recall the notion of {\em parabolic subderivative} for a proper function $f$ at $x\in\operatorname{dom}f$ relative to a vector $d$ for which $df(x)(d)$ is finite. It is defined by
\begin{align}\label{parab-subder}
d^{2}f(x)(d,w):=\liminf_{{t\downarrow 0}\atop{w'\rightarrow w}}\frac{f(x+td+\frac{1}{2}t^{2}w')-f(x)-tdf(x)(d)}{\frac{1}{2}t^{2}},\quad w\in\mathcal{R}^m.
\end{align}
The function $f$ is called {\em parabolically epi-differentiable} at $x$
for $d$ if $\operatorname{dom}d^{2}f(x)(d,\cdot)\neq\emptyset$
and for every $w\in\mathcal{R}^{m}$ and every $t_{k}\downarrow 0$, there exists $w_{k}\rightarrow w$ such that 
\begin{align}\label{parab-epi}
d^{2}f(x)(d,w)=\lim_{k\rightarrow\infty}\frac{f(x+t_{k}d+\frac{1}{2}t_{k}^{2}w_{k})-f(x)-t_{k}df(x)(d)}{\frac{1}{2}t_{k}^{2}}.
\end{align}

Next we define the fundamental second-order regularity concept introduced in \cite{Rockafellar1998} and comprehensively developed in the recent papers  \cite{Mohammadi2020,Mohammadi2021,Mohammadi2022}, which revealed the fulfillment of this property for broad classes of functions appearing in variational analysis and optimization.

\begin{definition} Given $f:\mathcal{R}^{m}\rightarrow \overline{\mathcal{R}}$ with $x\in\operatorname{dom}f$ and $u\in\partial f(x)$, we say that $f$ is {\sc parabolically regular} at $x$ for $u$ in the direction $d$ if there is the connection
\begin{align*}
\inf_{w\in\mathcal{R}^{m}}\left\{d^{2}f(x)(d,w)-\langle u,w\rangle\right\}=d^{2}f(x,u)(d)
\end{align*} 
between the second subderivative \eqref{2subder} and the parabolic subderivative \eqref{parab-subder}.
\end{definition}

Now we are ready to formulate our basic general assumptions, which will be applied below to the initial data of the composite optimization problem \eqref{P}.
 
\begin{assumption}\label{assump-1}
Given an l.s.c.\ proper convex function $f:\mathcal{R}^{m}\rightarrow\overline{\mathcal{R}}$, assume that $f$ is {\sc parabolically regular} at $\ox\in\dom f$ for $\bar{u}\in\partial f(\bar{x})$ and {\sc parabolically epi-differentiable} at $\bar{x}$ for every direction $d\in\mathcal{R}^m$ satisfying the condition $df(\bar{x})(d)=\langle\bar{u},d\rangle$. 
\end{assumption}

\begin{assumption}\label{assump-1-1}
Given an l.s.c.\ proper convex function $f:\mathcal{R}^{m}\rightarrow \overline{\mathcal{R}}$, assume that the subgradient mapping $\partial f$ is {\sc calm} at $\ox\in\dom f$ for $\bar{u}\in\partial f(\bar{x})$.
\end{assumption}

Let us show that both Assumptions~\ref{assump-1} and \ref{assump-1-1} are mild indeed. To this end, recall that a closed convex {\em set} $C\subset\mathcal{R}^{m}$ is {\em ${\cal C}^{2}$-reducible} to some $K\subset\mathcal{R}^{p}$ at $x\in C$ if there exist a neighborhood $U$ of $x$ and a ${\cal C}^2$-smooth  mapping $\Xi:U\rightarrow\mathcal{R}^{p}$ with the surjective derivative such that $C\cap U=\{x\in U\, |\, \Xi(x)\in K\}$ and $\Xi(x)=0$. The reduction is {\em pointed} if the tangent cone $\mathcal{T}_{K}(\Xi(x))$ is pointed. If in addition the set $K-\Xi(x)$ is a pointed closed convex cone, it is said that $C$ is ${\cal C}^2$-cone reducible at $x$. We say that an l.s.c.\  convex {\em function} $f:\mathcal{R}^{m}\to\overline{\mathcal{R}}$ is {\em ${\cal C}^{2}$-cone reducible} at $x\in\dom f$ if its epigraph is ${\cal C}^{2}$-cone reducible at $(x,f(x))$. Moreover, $f$ is said to be simply ${\cal C}^{2}$-cone reducible if it is ${\cal C}^{2}$-cone reducible at every point of the domain. 

It follows from \cite[Theorem~6.2]{Mohammadi2021} that the ${\cal C}^2$-reducibility of the function $f$ yields its parabolic epi-differentiability at $\bar{x}$ for $\bar{u}$ in \eqref{parab-epi} as required by Assumption~\ref{assump-1}, while the parabolic regularity of such $f$ is a consequence of \cite[Propositions~3.103, 3.136]{Bonnans2000}. The subdifferential calmness of ${\cal C}^2$-reducible functions required in Assumption~\ref{assump-1-1} is verified in the following proposition.

\begin{proposition}\label{prop-calmness-C2-cone-reducible-function}
Let $f:\mathcal{R}^{m}\rightarrow\overline{\mathcal{R}}$ be a proper convex ${\cal C}^{2}$-cone reducible function, which is Lipschitz continuous around each point in $\operatorname{dom}f$. Given $\bar{x}\in\operatorname{dom}f$ and $\bar{u}\in\partial f(\bar{x})$, the subgradient mapping $\partial f$ is calm at $\bar{x}$ for $\bar{u}$.
\end{proposition}
\begin{proof}
Since $\operatorname{epi}f$ is a closed convex set that is ${\cal C}^{2}$-cone reducible at $(\bar{x},f(\bar{x}))$, it follows from \cite[Theorem~2.1]{LiuSunPan2019} that the normal cone mapping $\mathcal{N}_{\operatorname{epi} f}$ is calm at $(\bar{x},f(\bar{x}))$ for each $(u,\gamma)\in\mathcal{N}_{\operatorname{epi} f}(\bar{x},f(\bar{x}))$. Furthermore, we get by \cite[Proposition~1.17]{Mordukhovich2018} that
\begin{align}\label{eq:prop-3-normal-formular}
\mathcal{N}_{\operatorname{epi} f}(\bar{x},f(\bar{x}))=\big\{\lambda(v,-1)\, \big|\, v\in\partial f(\bar{x}),\;\lambda>0\big\}\cup\big\{(v,0)\, \big|\, v\in\partial^{\infty}f(\bar{x})\big\},
\end{align}
which yields $(\bar{u},-1)\in\mathcal{N}_{\operatorname{epi} f}(\bar{x},f(\bar{x}))$ and  therefore implies that the normal cone mapping $\mathcal{N}_{\operatorname{epi} f}$ is calm at $(\bar{x},f(\bar{x}))$ for $(\bar{u},-1)$. The latter tells us that there exist numbers $\kappa,\varepsilon,\delta>0$ such that for all $(x,\alpha)\in\mathbb{B}((\bar{x},f(\bar{x})),\varepsilon)\cap\operatorname{epi}f$, we have the inclusion
\begin{align*}
\mathcal{N}_{\operatorname{epi} f}(x,\alpha)\cap\mathbb{B}((\bar{u},-1),\delta)\subset\mathcal{N}_{\operatorname{epi}f}(\bar{x},f(\bar{x}))+\kappa\|(x,\alpha)-(\bar{x},f(\bar{x}))\|\mathbb{B}_{\mathcal{R}^{m+1}}.
\end{align*}
Moreover, choosing $\varepsilon$ and $\delta$ to be sufficiently small leads us to
\begin{align}\label{eq:prop-3-empty-intersection}
\mathbb{B}((\bar{u},-1),\delta)\cap\big(\big\{(v,0)\, \big|\, v\in\partial^{\infty}f(\bar{x})\big\}+\kappa\|(x,\alpha)-(\bar{x},f(\bar{x}))\|\mathbb{B}_{\mathcal{R}^{m+1}}\big)=\emptyset.
\end{align}
Since $f$ is continuous at $\bar{x}$, there exists $\delta'>0$ such that for any $x\in\mathbb{B}(\bar{x},\delta')\cap\operatorname{dom}f$, we get
$|f(x)-f(\bar{x})|\le\varepsilon/2$. Letting $\varepsilon':=\min\{\frac{\varepsilon}{2},\delta'\}$ and taking any $x\in\mathbb{B}(\bar{x},\varepsilon')$ ensure that $(x,f(x))\in\mathbb{B} ((\bar{x},f(\bar{x})),\varepsilon)\cap\operatorname{epi}f$ and that
\begin{align*}
\mathcal{N}_{\operatorname{epi}f}(x,f(x))\cap\mathbb{B}((\bar{u},-1),\delta)\subset\mathcal{N}_{\operatorname{epi}f}(\bar{x},f(\bar{x}))+\kappa\|(x,f(x))-(\bar{x},f(\bar{x}))\|\mathbb{B}_{\mathcal{R}^{m+1}}.
\end{align*}
Now we combine \eqref{eq:prop-3-normal-formular} and \eqref{eq:prop-3-empty-intersection} to obtain the inclusions
\begin{align*}
&\big\{\lambda(v,-1)\, \big|\, v\in\partial f(x),\lambda>0\big\}\cap\mathbb{B}((\bar{u},-1),\delta)\ \subset\ \mathcal{N}_{\operatorname{epi}f}(x,f(x))\cap\mathbb{B}((\bar{u},-1),\delta)\\
&\qquad\qquad\subset\ \big\{\lambda(v,-1)\, \big|\, v\in\partial f(\bar{x}),\lambda>0\big\}+\kappa\|(x,f(x))-(\bar{x},f(\bar{x}))\|\mathbb{B}_{\mathcal{R}^{m+1}},
\end{align*}
which readily allow us to conclude that
\begin{align*}
\partial f(x)\cap\mathbb{B}(\bar{u},\delta)&\subset\partial f(\bar{x})+\kappa\left(\|x-\bar{x}\|+|f(x)-f(\bar{x})|\right)\mathbb{B}_{\mathcal{R}^{m}}\\&\subset\partial f(\bar{x})+\kappa(1+L_{f})\|x-\bar{x}\|\mathbb{B}_{\mathcal{R}^{m}},\;\mbox{for all}\;x\in\mathbb{B}(\bar{x},\varepsilon'),
\end{align*}
where $L_f$ stands for the Lipschitz constant of $f$ around $\ox$. This shows that the subgradient mapping $\partial f$ is calm at $\bar{x}$ for $\bar{u}$.
\end{proof}

Recall that the class of convex ${\cal C}^{2}$-cone reducible functions is fairly rich including, e.g., the $\ell_{p}$ norm functions (for $p = 1, 2, \infty)$, the nuclear norm function, the spectral norm function, and indicator functions of standard cones in conic programming. This confirms that both Assumptions~\ref{assump-1} and \ref{assump-1-1} are pretty natural and unrestrictive. \vspace*{0.03in}

The next assumption is more involved, but also covers a large territory in optimization theory and applications. In its formulation, we use the $B$-subdifferential \eqref{B-sub} of the Lipschitzian proximal mapping \eqref{prox} and the Moore-Penrose inverses of its elements.

\begin{assumption}\label{assump-2} Given an l.s.c.\ proper convex function $f:\mathcal{R}^{m}\rightarrow \overline{\mathcal{R}}$ with $\bar{x}\in\operatorname{dom}f$ and $\bar{u}\in\partial f(\bar{x})$, assume that there exists $\overline{W}\in\partial_{B}\operatorname{Prox}_{f}(\bar{x}+\bar{u})$ such that
\begin{align*}
\bigcup\limits_{V\in{\cal J}\operatorname{Prox}_{f}(\ox+\ou)}\operatorname{rge}V=\operatorname{rge}\overline{W}
\end{align*}
and for any $y\in\bigcup\limits_{V\in{\cal J}\operatorname{Prox}_{f}(\ox+\ou)}\operatorname{rge}V$ we have $y=\overline{W}\bar{z}$, where $\bar{z}$ satisfies the equalities
\begin{align*}
\min\limits_{{z\in\mathcal{R}^{m},y=Vz}\atop{V\in{\cal J}\operatorname{Prox}_{f}(\ox+\ou)}}\langle y,z-y\rangle=\langle y,\bar{z}-y\rangle=\langle y,(\overline{W}^{\dagger}-I)y\rangle.
\end{align*}
\end{assumption}
Assumption~\ref{assump-2} is a technical requirement, which is needed for our characterizations of tilt-stable minimizers in the general class of composite optimization problems \eqref{P}. This assumption describes the properties that make such characterizations possible. Its verification is rather straightforward if the form of ${\cal J}\operatorname{Prox}_{f}(\bar{x}+\bar{u})$ is known, which is often the case in typical problems of machine learning and statistics; see, e.g., \cite{Khanh2023}. Due to the large size of the paper, we confine ourselves to the detailed verification of Assumption~\ref{assump-2} and explicit calculations of all the data therein for the {\em matrix spectral norm function} whose great importance has been realized in various areas of numerical optimization, data science, and operations research; see, e.g., \cite{ChenLiuSunToh2016} and the references therein. The aforementioned calculation results are presented in Appendix~\ref{appendices}.\vspace*{-0.12in}

\section{The Second-Order Variational Function}\label{sec:2nd-order}\vspace*{-0.05in}

In this section, we introduce and investigate the novel second-order variational function, which allows us to establish pointbased characterizations of tilt stability for local minimizers in the general composite optimization problem \eqref{P} obtained in Section~\ref{sec:no-gap}. We'll see below that this function vanishes when \eqref{P} is a linear program in which case the obtained criteria reduce to those in \cite{GfrererMordukhovich2015}. Thus the second-order variational function can be viewed as a ``measure of nonpolyhedrality".

To begin with, recall that the {\em Lagrangian function} for problem \eqref{P} is defined by
\begin{align*}
\mathcal{L}(x,u):=\inf_{s\in\mathcal{R}^{m}}\Big\{f_{0}(x)+g(F(x)+s)-\langle u,s\rangle\Big\}=L(x,u)-g^{*}(u),
\end{align*}
where $L(x,u)=f_{0}(x)+\langle F(x),u\rangle$. The corresponding {\em Karush-Kuhn-Tucker} (KKT) conditions for \eqref{P} can be written as in the form
\begin{align}\label{eq:kkt-condition}
\nabla_{x}L(x,u)=0,\quad u\in\partial g(F(x)).
\end{align}
Given an extended-real-valued function $f:\mathcal{R}^{m}\rightarrow\overline{\mathcal{R}}$ with $x\in\operatorname{dom}f$ and $u\in\partial f(x)$, define the {\em second-order variational function} (SOVF) associated with $f$ by
\begin{align}\label{sovf}
\Gamma_{f}(x,u)(v):=\left\{\begin{array}{cl}\min\limits_{{d\in\mathcal{R}^{m},v=Vd}\atop{V\in{\cal J} \operatorname{Prox}_{f}(x+u)}}\langle v,d-v\rangle&\mbox{if}\ v\in\bigcup\limits_{V\in{\cal J}\operatorname{Prox}_{f}(x+u)}\operatorname{rge}V,\\
\infty&\mbox{otherwise},
\end{array}\right.
\end{align}
and establish its generalized differential properties, which are of their own interest while being used below for deriving the main results of this paper. We first recall the following relationship that is proved in \cite[Proposition~3.3]{Tangwang2025}.

\begin{proposition}\label{lemma:reformulate-dom-df(x,u)-polar}
Given an l.s.c.\ proper convex function $f:\mathcal{R}^{m}\rightarrow\overline{\mathcal{R}}$, $x\in\operatorname{dom}f$ and $u\in\partial f(x)$. Then we have
\begin{align*}
\big\{d\,\big|\,df(x)(d)=\langle u,d\rangle\big\}^{\circ}=\big\{d\,\big|\,df^{*}(u)(d)=\langle x,d\rangle\big\}\cap\big\{d\,\big|\,d^{2}f^{*}(u,x)(d)=0\big\}.
\end{align*}
\end{proposition}

Observe that if $f:=\delta_{C}$  and Assumption~\ref{assump-1} is satisfied at $x$ for $u$, then Proposition~\ref{lemma:reformulate-dom-df(x,u)-polar} reduces to \cite[Corollary~3.1]{MohammadBoris2020}.\vspace*{0.05in}

The next proposition provides a detailed description of the {\em SOVF domain} in \eqref{sovf}. 

\begin{proposition}\label{proposition-domain-Gamma-f} Given an extended-real-valued function $f:\mathcal{R}^{m}\rightarrow\overline{\mathcal{R}}$, suppose that $f$ satisfies Assumption~{\rm\ref{assump-1}} at some $x\in\operatorname{dom}f$ for $u\in\partial f(x)$. Then $f$ is properly twice epi-differentiable at $x$ for $u$ with $\operatorname{dom}d^{2}f(x,u)=\{d\,|\,df(x)(d)=\langle u,d\rangle\}$. Moreover, for any $v\in\operatorname{dom}d^{2}f(x,u)$ and $z\in D(\partial f)(x,u)(v)$, we have the equalities
\begin{align}\label{eq:d2f-d2f*-value}
d^{2}f(x,u)(v)=d^{2}f^{*}(u,x)(z)=\langle v,z\rangle.
\end{align}
If in addition $f$ satisfies Assumption~{\rm\ref{assump-2}} at $x$ for $u$, then the domain of SOVF is described by
\begin{align}\label{domain-gamma}
\operatorname{dom}\Gamma_{f}(x,u)=\operatorname{aff}\big\{d\,\big|\,df(x)(d)=\langle u,d\rangle\big\}=\operatorname{dom}D^{*}(\partial f)(x,u).
\end{align}
\end{proposition}
\begin{proof} By Assumption~\ref{assump-1}, we deduce from \cite[Theorem~3.8]{Mohammadi2020} that the function $f$ is properly twice epi-differentiable at $x$ for $u$ with the domain $\operatorname{dom}d^{2}f(x,u)$ defined as $\left\{d\, \left|\, df(x)(d)=\langle u,d\rangle\right.\right\}$. Consequently, \cite[Lemma~3.6]{Chieu2021} yields the equality $d^{2}f(x,u)(v)=\langle v,z\rangle$. Employing  \cite[Theorem~13.21]{Rockafellar1998} and \cite[Theorem~23.5]{Rockafellar1970} leads us to
\begin{align*}
\frac{1}{2}d^{2}f(x,u)(v)+\frac{1}{2}d^{2}f^{*}(u,x)(z)=\langle v,z\rangle,
\end{align*}
which completes the verification of \eqref{eq:d2f-d2f*-value}.

We know by the convexity of the l.s.c.\ function $f$ that it is prox-regular and subdifferentially continuous at every point in $\operatorname{dom}f$. Taking into account the twice epi-differentiability of $f$ and using \eqref{eq:inverse-graphical-derivative}, \eqref{eq:coderivative-sum}, Assumption~\ref{assump-1} together with \cite[Theorem~13.57]{Rockafellar1998} allows us 
to conclude that
\begin{align*}
w\in z+D(\partial f)(x,u)(z)\subset z+D^{*}(\partial f)(x,u)(z)=D^{*}(I+\partial f)(x,x+u)(z)		
\end{align*}
for all $w\in\mathcal{R}^{m}$ and $z\in D\operatorname{Prox}_{f}(x+u)(w)$. Consequently, we have
$-z\in D^{*}\operatorname{Prox}_{f}(x+u)(-w)$,
which implies in turn that 
\begin{align*}
\operatorname{rge}D\operatorname{Prox}_{f}(x+u)\subset-\operatorname{rge}D^{*}\operatorname{Prox}_{f}(x+u).
\end{align*}
Invoking now Assumption~\ref{assump-2} and \cite[Lemma~4.2]{Tangwang2025} tells us that  
\begin{align*}
\operatorname{aff}\operatorname{dom}D(\partial f)(x,u)=\operatorname{aff}\operatorname{rge}D\operatorname{Prox}_{f}(x+u)=\operatorname{rge}D^{*}\operatorname{Prox}_{f}(x+u).
\end{align*}
Combining the latter with  \cite[Proposition~2.4]{Gfrerer2022}, and \cite[Corollary~13.53]{Rockafellar1998} yields
\begin{align*}
\begin{aligned}
&\operatorname{rge}\overline{W}\subset\operatorname{rge}D^{*}\operatorname{Prox}_{f}(x+u)\subset \bigcup\limits_{V\in{\cal J}\operatorname{Prox}_{f}(x+u)}\operatorname{rge}V=\operatorname{rge}\overline{W},
\end{aligned}
\end{align*}
which justifies \eqref{domain-gamma} by \cite[Lemma~5.2]{MordukhovichTangWang2025} and thus completes the proof of the proposition.
\end{proof}

Under the assumptions of the proposition above, we get the alternative formulation of the SOVF $\Gamma_{f}(x,u)$ as follows, which greatly simplifies its calculation and applications.

\begin{proposition}\label{proposition-reformulate-Gamma-f} Let $f$ satisfy all the assumptions of Proposition~{\rm\ref{proposition-domain-Gamma-f}}.
Then for any $v\in\operatorname{dom}\Gamma_{f}(x,u)$, we have the representations
\begin{align}\label{sovf-rep}
\Gamma_{f}(x,u)(-v)=\Gamma_{f}(x,u)(v),\quad
\Gamma_{f}(x,u)(v)=\min_{z\in D^{*}(\partial f)(x,u)(v)}\langle v,z\rangle.
\end{align}
\end{proposition}
\begin{proof} The first formula in \eqref{sovf-rep} follows immediately from definition \eqref{sovf}. To verify the second one, we deduce from 
Proposition~\ref{proposition-domain-Gamma-f} that
\begin{align*}
\operatorname{dom}\Gamma_{f}(x,u)=\operatorname{dom}D^{*}(\partial f)(x,u)=\operatorname{rge}D^{*}\operatorname{Prox}_{f}(x+u).
\end{align*}
Picking any $v\in\operatorname{dom}\Gamma_{f}(x,u)$ and $z\in D^{*}(\partial f)(x,u)(v)$ gives us by \eqref{eq:inverse-graphical-derivative} and \eqref{eq:coderivative-sum} that
\begin{align*}
\langle v,z\rangle=\langle -v,w-(-v)\rangle\;\mbox{ with }\;w=-v-z,\,-v\in D^{*}\operatorname{Prox}_{f}(x+u)(w),
\end{align*}
which yields therefore the equality
\begin{align*}
\min_{z\in D^{*}(\partial f)(x,u)(v)}\langle v,z\rangle=\min_{-v\in D^{*}\operatorname{Prox}_{f}(x+u)(w)}\langle -v,w-(-v)\rangle.
\end{align*}
On one hand, it follows from  \cite[Corollary~13.53]{Rockafellar1998} the lower estimate
\begin{align*}
\Gamma_{f}(x,u)(-v)\leq\min_{-v\in D^{*}\operatorname{Prox}_{f}(x+u)(w)}\langle -v,w-(-v)\rangle.
\end{align*}
On the other hand,  \cite[Proposition~2.4]{Gfrerer2022} and Assumption~\ref{assump-2} tell us that
\begin{align*}
\min_{-v\in D^{*}\operatorname{Prox}_{f}(x+u)(w)}\langle -v,w-(-v)\rangle\leq\langle -v,d-(-v)\rangle=\Gamma_{f}(x,u)(-v),\quad\mbox{with}\ -v=\overline{W}d,
\end{align*}
which thus completes the proof of the proposition.
\end{proof}

Observe that under Assumptions~\ref{assump-1} and \ref{assump-2}, the function $\Gamma_{f}(x,u)$ is {\em generalized quadratic} with its domain being a {\em linear subspace}. This holds, e.g., when $f$ is a locally Lipschitz continuous and {\em 
${\cal C}^{2}$-cone reducible} convex function. Particular examples include the nuclear norm, the spectral norm, and the indicator functions of standard cones in conic programming such as the positive orthant cone, the second-order cone, and the positive semidefinite cone.\vspace*{0.05in}

Next we derive the following relationship between the ranges of the regular coderivative and the graphical derivative for the proximal mapping associated with a given function.

\begin{lemma}\label{lemma:subseq-range-regular-coderivative}
Let $f$ be an extended-real-valued function with $x\in\operatorname{dom}f$ and $u\in\partial f(x)$. Suppose that $f$ satisfies Assumption~{\rm\ref{assump-1}} at $x$ for $u$. Then we have the inclusion
\begin{align*}
\operatorname{rge}\widehat{D}^{*}\operatorname{Prox}_{f}(x+u)\subset\operatorname{rge}D\operatorname{Prox}_{f}(x+u).
\end{align*}
\end{lemma}
\begin{proof} It follows from \eqref{eq:inverse-graphical-derivative},  Assumption~\ref{assump-1}, and \cite[Lemma~5.2]{MordukhovichTangWang2025} that
\begin{align}\label{eq:range-graphical-proximal}
\operatorname{rge}D\operatorname{Prox}_{f}(x+u)=\operatorname{dom}D(\partial f)(x,u)=\left\{d\, \left|\, df(x)(d)=\langle u,d\rangle\right.\right\}.
\end{align}
Employing the conditions given in \cite[Corollary~10.11]{Rockafellar1998} yields
\begin{align*}
\operatorname{rge}\widehat{D}^{*}\operatorname{Prox}_{f}(x+u)&\subset\widehat{\mathcal{N}}_{(I+\partial f)(x)}(x+u)=\mathcal{N}_{(I+\partial f)(x)}(x+u)\nonumber\\
&=\mathcal{N}_{\partial f(x)}(u)=\left\{d\, \left|\, df(x)(d)=\langle u,d\rangle\right.\right\},
\end{align*}
which being combined with \eqref{eq:range-graphical-proximal} verifies the claimed inclusion.
\end{proof}

Finally in this section, we establish important relationships between the 
second-order function and second subderivative under the imposed assumptions.

\begin{theorem}\label{proposition-inequality-Gamma-f}
Let $f$ be an extended-real-valued function with $x\in\operatorname{dom}f$ and $u\in\partial f(x)$. The following assertions hold: 

{\bf(i)} Suppose that $f$ satisfies Assumption~{\rm\ref{assump-1}} at $x$ for $u$. Then for any $v\in\operatorname{dom
}D^{*}(\partial f)(x,u)$, there exists  sequences $(x^{k},u^{k})\rightarrow(x,u)$ with $u^{k}\in\partial f(x^{k})$ and $v^{k}\rightarrow v$ as $k\to\infty$ such that
\begin{align}\label{ineq-Gamma-f-less-cond}
\min_{v^{*}\in D^{*}(\partial f)(x,u)(v)}\langle v^{*},v\rangle\geq\limsup_{k\rightarrow\infty}d^{2}f(x^{k},u^{k})(v^{k}).
\end{align}
If in addition $f$ satisfies Assumption~{\rm\ref{assump-2}} at $x$ for $u$, then we have
\begin{align}\label{ineq-Gamma-f-1}
\Gamma_{f}(x,u)(v)\geq\limsup_{k\rightarrow\infty}d^{2}f(x^{k},u^{k})(v^{k}).
\end{align} 

{\bf(ii)} If both Assumptions~{\rm\ref{assump-1}} and {\rm\ref{assump-2}} are satisfied, then for any $v\in\operatorname{dom}\Gamma_{f}(x,u)$ and any sequence $(\tilde{x}^{k},\tilde{u}^{k})\rightarrow(x,u)$ with $\tilde{u}^{k}\in\partial f(\tilde{x}^{k})$ and $\tilde{v}^{k}\rightarrow v$, it follows that
\begin{align}\label{ineq-Gamma-f-0}
\Gamma_{f}(x,u)(v)\leq\liminf_{k\rightarrow\infty}d^{2}f(\tilde{x}^{k},\tilde{u}^{k})(\tilde{v}^{k}).
\end{align}
    
{\bf(iii)} When $f$ is ${\cal C}^{2}$-cone reducible and satisfies Assumption~{\rm\ref{assump-1}} in a neighborhood of $x$  and Assumption~{\rm\ref{assump-2}} at $x$ for $u$, then we get for any $v\in\operatorname{dom}D^{*}(\partial f)(x,u)$ with $(v,\beta)\in\operatorname{dom}D^{*} \mathcal{N}_{\operatorname{epi} f}((x,f(x)),(u,-1))$ that
\begin{align}\label{ineq-Gamma-f-2}
\min_{v^{*}\in D^{*}(\partial f)(x,u)(v)}\langle v^{*},v\rangle=\min_{(z,\gamma)\in D^{*}\mathcal{N}_{\operatorname{epi}f}((x,f(x)),(u,-1))(v,\beta)}\langle(z,\gamma),(v,\beta)\rangle.
\end{align}
\end{theorem}
\begin{proof}
It follows from \eqref{eq:inverse-graphical-derivative}, \eqref{eq:coderivative-sum}, and Lemma~\ref{lemma:subseq-range-regular-coderivative} that
\begin{align}
\label{eq:dom-regular-coderivative-dom-graphical-derivative}
\operatorname{dom}\widehat{D}^{*}(\partial f)(x,u)\subset-\operatorname{dom}D(\partial f)(x,u).
\end{align}  
Pick any $z\in D^{*}(\partial f)(x,u)(v)$ and find $(x^{k},u^{k})$ with $u^{k}\in\partial f(x^{k})$ and $(z^{k},v^{k})$ with $z^{k}\in\widehat{D}^{*}(\partial f)(x^{k},u^{k})(-v^{k})$ such that $(x^{k},u^{k})\rightarrow(x,u)$ and $(z^{k},-v^{k})\rightarrow(z,v)$ as $k\to\infty$, which yields $D(\partial f)(x^{k},u^{k})(v^{k})\neq\emptyset$. Using  \cite[Proposition~8.37]{Rockafellar1998} and \cite[Lemma~3.6]{Chieu2021} tells us that
\begin{align*}
\langle z^{k},-v^{k}\rangle\geq\langle v^{k},w^{k}\rangle=d^{2}f(x^{k},u^{k})(v^{k}),
\end{align*}
where $w^{k}\in D(\partial f)(x^{k},u^{k})(v^{k})$. Passing to the limit as $k\rightarrow\infty$, we arrive at \eqref{ineq-Gamma-f-less-cond}. Consequently, applying \eqref{sovf-rep} of 
Proposition~\ref{proposition-reformulate-Gamma-f} brings us to \eqref{ineq-Gamma-f-1} and thus justifies assertion (i). 
	
To proceed further, for any $(\tilde{x}^{k},\tilde{u}^{k})\rightarrow(x,u)$ with $\tilde{u}^{k}\in\partial f(\tilde{x}^{k})$, $\tilde{v}^{k}\rightarrow v$, observe that
$\Gamma_{f}(\tilde{x}^{k},\tilde{u}^{k})(\tilde{v}^{k})\leq d^{2}f(\tilde{x}^{k},\tilde{u}^{k})(\tilde{v}^{k})$, and hence
\begin{align*}
\limsup_{k\rightarrow\infty}\Gamma_{f}(\tilde{x}^{k},\tilde{u}^{k})(\tilde{v}^{k})\leq\liminf_{k\rightarrow\infty}d^{2}f(\tilde{x}^{k},\tilde{u}^{k})(\tilde{v}^{k}).
\end{align*}
There exists a sequence of $(\hat{x}^{k},\hat{u}^{k})\rightarrow(x,u)$ such that $\operatorname{Prox}_{f}$ is differentiable at $\hat{x}^{k}+\hat{u}^{k}$ and that $W^{k}=\nabla\operatorname{Prox}_{f}(\hat{x}^{k}+\hat{u}^{k})$ with 
$W^{k}\rightarrow\overline{W}\in\partial_{B}\operatorname{Prox}_{f}(x+u)$ with $\Gamma_{f}(\hat{x}^{k},\hat{u}^{k})(\tilde{v}^{k})=\langle\tilde{v}^{k},((W^{k})^{\dagger}-I)\tilde{v}^{k}\rangle$. Moreover, we get $\tilde{v}^{k}\in\operatorname{rge}\overline{W}$ for all $k\in\mathbb N$ sufficiently large with the equalities
\begin{align*}
\lim_{k\rightarrow\infty}\Gamma_{f}(\hat{x}^{k},\hat{u}^{k})(\tilde{v}^{k})=\langle v,(\overline{W}^{\dagger}-I)v\rangle=\Gamma_{f}(x,u)(v).
\end{align*}
This readily implies that
\begin{align*}
\Gamma_{f}(x,u)(v)&\leq\limsup_{k\rightarrow\infty}\Gamma_{f}(\tilde{x}^{k},\tilde{u}^{k})(\tilde{v}^{k})\leq\liminf_{k\rightarrow\infty} d^{2}f(\tilde{x}^{k},\tilde{u}^{k})(\tilde{v}^{k}),
\end{align*}  
which justifies the estimate in \eqref{ineq-Gamma-f-0} of assertion (ii).
		
Now we verify equality \eqref{ineq-Gamma-f-2} in assertion (iii). Given $\lambda>0$, $\tilde{x}\in\operatorname{dom}f$, $\tilde{\alpha}:=f(\tilde{x})$, and $\tilde{u}\in\partial f(\tilde{x})$, use \cite[Lemma~5.2]{MordukhovichTangWang2025},
\cite[Theorem~8.2]{Rockafellar1998}, and \cite[Proposition~2.126]{Bonnans2000} to see that the domain of $D\mathcal{N}_{\operatorname{epi}f}((\tilde{x},\tilde{\alpha}),\lambda(\tilde{u},-1))$ admits the representations
\begin{align}\label{eq:dom-graphical-normal-epi-g}
\begin{array}{ll}
\operatorname{dom}D\mathcal{N}_{\operatorname{epi}f}((\tilde{x},\tilde{\alpha}),\lambda(\tilde{u},-1))
=\operatorname{dom}D(\partial\delta_{\operatorname{epi}f}((\tilde{x},\tilde{\alpha}),\lambda(\tilde{u},-1)))\\
=\left\{(y,c)\,\left|\,d\delta_{\operatorname{epi}f}(\tilde{x},\tilde{\alpha})(y,c)=\lambda\langle\tilde{u},y\rangle-\lambda c\right.\right\}
=\left\{(y,c)\,\left|\,df(\tilde{x})(y)=\langle \tilde{u},y\rangle=c\right.\right\}\\
=\{(y,c)\,|\,y\in\operatorname{dom}D(\partial f)(\tilde{x},\tilde{u}),\,c
=\langle\tilde{u},y\rangle\}
:=\widehat{\mathcal{K}}((\tilde{x},\tilde{\alpha}),\lambda(\tilde{u},-1)).
\end{array}
\end{align}	
It follows from \eqref{eq:inverse-graphical-derivative},  \eqref{eq:coderivative-sum}, and Lemma~\ref{lemma:subseq-range-regular-coderivative} that
\begin{align}\label{eq:dom-regular-coderivative-dom-graphical-derivative-epi}
-\operatorname{dom}\widehat{D}^{*}\mathcal{N}_{\operatorname{epi}f}((\tilde{x},\tilde{\alpha}),\lambda(\tilde{u},-1))&\subset\operatorname{dom}D\mathcal{N}_{\operatorname{epi}f}((\tilde{x},\tilde{\alpha}),\lambda(\tilde{u},-1)).
\end{align}
Since the classical Robinson's constraint qualification clearly holds for the constraint $(x,\gamma)\in\operatorname{epi}f$, the set-valued mapping $M(x,\gamma):=(x,\gamma)-\operatorname{epi}f$ is metrically subregular at $(\tilde{x},\tilde{\alpha})$ for $(0,0)$. Moreover,  the set $\operatorname{epi}f$ is ${\cal C}^{2}$-cone reducible at $(\tilde{x},\tilde{\alpha})$ with $h$ being the corresponding  ${\cal C}^2$-smooth mapping. By \cite[Corollary~5.4]{Helmut2019}, for any $(y,c)\in\widehat{\mathcal{K}} ((\tilde{x},\tilde{\alpha}),\lambda(\tilde{u},-1))$ we have
\begin{align*}
D\mathcal{N}_{\operatorname{epi}f}((\tilde{x},\tilde{\alpha}),\lambda(\tilde{u},-1))(y,c)=&\lambda\nabla^{2}\left\langle(\nabla h(\tilde{x},\tilde{\alpha})^{T})^{-1}(\tilde{u},-1),h(\cdot,\cdot)\right\rangle(\tilde{x},\tilde{\alpha})(y,c)\\
&+\mathcal{N}_{\widehat{\mathcal{K}}((\tilde{x},\tilde{\alpha}),\lambda(\tilde{u},-1))}(y,c).
\end{align*}
Given $(v,\beta)\in\operatorname{dom}\widehat{D}^{*}\mathcal{N}_{\operatorname{epi}f}((\tilde{x},\tilde{\alpha}),\lambda(\tilde{u},-1))$ and 
\begin{align*}
(v^{*},\beta^{*})\in\lambda\nabla^{2}\left\langle(\nabla h(\tilde{x},\tilde{\alpha})^{T})^{-1}(\tilde{u},-1),h(\cdot,\cdot)\right\rangle(\tilde{x},\tilde{\alpha})(v,\beta)+\widehat{\mathcal{K}}((\tilde{x},\tilde{\alpha}),\lambda(\tilde{u},-1))^{\circ},
\end{align*}
we deduce from \eqref{eq:dom-regular-coderivative-dom-graphical-derivative-epi} that $(v,\beta)\in-\widehat{\mathcal{K}}
((\tilde{x},\tilde{\alpha}),\lambda(\tilde{u},-1))$ and
\begin{align*}
\left\langle((v^{*},\beta^{*}),-(v,\beta)),((y,c),(y^{*},c^{*}))\right\rangle\leq 0,\quad\forall\,(y^{*},c^{*})\in D\mathcal{N}_{\operatorname{epi}f}((\tilde{x},\tilde{\alpha}),\lambda(\tilde{u},-1))(y,c),
\end{align*}
which yields in turn the inclusion
\begin{align*}
&\lambda\nabla^{2}\left\langle(\nabla h(\tilde{x},\tilde{\alpha})^{T})^{-1}(\tilde{u},-1),h(\cdot,\cdot)\right\rangle(\tilde{x},\tilde{\alpha})(v,\beta)
+\widehat{\mathcal{K}}((\tilde{x},\tilde{\alpha}),\lambda(\tilde{u},-1))^{\circ}\\
&\qquad\subset \widehat{D}^{*}\mathcal{N}_{\operatorname{epi}f}((\tilde{x},\tilde{\alpha}),\lambda(\tilde{u},-1))(v,\beta).
\end{align*}
By the limiting coderivative construction combined with the outer semicontinuity of the normal cone mapping (see, e.g., \cite[Proposition~6.5]{Rockafellar1998}), we get that
\begin{align}\label{eq:relation-normal-gph-normal-epi-g}
\begin{array}{ll}
&\disp\mathcal{N}_{\operatorname{gph}\mathcal{N}_{\operatorname{epi}f}}((x,f(x)),(u,-1))\\
&=\disp\Limsup_{((\tilde{x},\tilde{\alpha}),\lambda(\tilde{u},-1)){\xrightarrow{\operatorname{gph}\mathcal{N}_{\operatorname{epi}f}}}((x,f(x)),\disp(u,-1))}\widehat{\mathcal{N}}_{\operatorname{gph}\mathcal{N}_{\operatorname{epi}f}}((\tilde{x},\tilde{\alpha}),\lambda(\tilde{u},-1))\\
&\disp\supset\Limsup_{((\tilde{x},\tilde{\alpha}),\lambda(\tilde{u},-1)){\xrightarrow{\operatorname{gph}\mathcal{N}_{\operatorname{epi}f}}}((x,f(x)),(u,-1))}G((\tilde{x},\tilde{\alpha}),\lambda(\tilde{u},-1)):=\mathcal{G}, 
\end{array}
\end{align}
where the mapping $G(\cdot)$ under the outer limit \eqref{pk} in \eqref{eq:relation-normal-gph-normal-epi-g} is defined by
\begin{align*}
G((\tilde{x},\tilde{\alpha}),\lambda(\tilde{u},-1))&:=\big\{((v^{*},\beta^{*}),-(v,\beta))\,\big|\,(v^{*},\beta^{*})\in\\
&\lambda\nabla^{2}\left\langle(\nabla h(\tilde{x},\tilde{\alpha})^{T})^{-1}(\tilde{u},-1),h(\cdot,\cdot)\right\rangle(\tilde{x},\tilde{\alpha})(v,\beta)\\
&
+\widehat{\mathcal{K}}((\tilde{x},\tilde{\alpha}),\lambda(\tilde{u},-1))^{\circ},\,(v,\beta)\in\operatorname{dom}\widehat{D}^{*}\mathcal{N}_{\operatorname{epi}f}((\tilde{x},\tilde{\alpha}),\lambda(\tilde{u},-1))\big\}.
\end{align*}
The relation \eqref{eq:dom-regular-coderivative-dom-graphical-derivative} implies that for any $v\in\operatorname{dom}\widehat{D}^{*}(\partial f)(\tilde{x},\tilde{u})$, we have $-v\in\operatorname{dom}D(\partial f)(\tilde{x},\tilde{u})$. Using  Proposition~\ref{proposition-domain-Gamma-f} and the equalities
\begin{align*}
\Limsup_{(\tilde{x},\tilde{u}){\xrightarrow{\operatorname{gph}\partial f}}(x,u)}\operatorname{dom}\widehat{D}^{*}(\partial f)(\tilde{x},\tilde{u})=\operatorname{dom}D^{*}(\partial f)(x,u)=\operatorname{aff}\big\{d\,\big|\,df(x)(d)=\langle u,d\rangle\big\}
\end{align*}
combined with \eqref{eq:dom-regular-coderivative-dom-graphical-derivative}, \eqref{eq:dom-graphical-normal-epi-g} and \eqref{eq:relation-normal-gph-normal-epi-g} ensures that 
\begin{align*}
\operatorname{dom}D^{*}(\partial f)(x,u)\subset\Limsup_{(\tilde{x},\tilde{u}){\xrightarrow{\operatorname{gph}\partial f}}(x,u)}\operatorname{dom}D(\partial f)(\tilde{x},\tilde{u}).
\end{align*}
Hence for any $v\in\operatorname{dom}D^{*}(\partial f)(x,u)$, there exist $(v^{*},\beta^{*})\in\mathcal{R}^{m+1}$ and $\beta\in\mathcal{R}$ such that
\begin{align*}
(v,\beta)\in\operatorname{dom}D^{*}\mathcal{N}_{\operatorname{epi}f}((x,f(x)),(u,-1))\;\mbox{ and }\;
((v^{*},\beta^{*}),-(v,\beta))\in\mathcal{G}.
\end{align*}
It follows from \eqref{eq:relation-normal-gph-normal-epi-g} that for any $v\in\operatorname{dom}D^{*}(\partial f)(x,u)$ and $\beta\in\mathcal{R}$ satisfying $(v,\beta)\in\\\operatorname{dom} D^{*}\mathcal{N}_{\operatorname{epi}f}((x,f(x)),(u,-1))$, we get the conditions
\begin{align}\label{eq:inequlity-minimal-coderivative-1}
\begin{array}{ll}
\disp\min_{(z,\gamma)\in D^{*}\mathcal{N}_{\operatorname{epi}f}((x,f(x)),(u,-1))(v,\beta)}\langle(v,\beta), (z,\gamma)\rangle\\
=\disp\min_{((z,\gamma),-(v,\beta))\in \disp\mathcal{N}_{\operatorname{gph}\mathcal{N}_{\operatorname{epi}f}}((x,f(x)),(u,-1))}\langle(v,\beta), (z,\gamma)\rangle\\
\disp\leq\min_{((v^{*},\beta^{*}),-(v,\beta))\in\mathcal{G}}\langle(v^{*},\beta^{*}), (v,\beta)\rangle\\
=\disp\liminf_{((\tilde{x},\tilde{\alpha}),\lambda(\tilde{u},-1)){\xrightarrow{\operatorname{gph}\mathcal{N}_{\operatorname{epi}f}}}((x,f(x)),(u,-1))} \min_{((\tilde{v}^{*},\tilde{\beta}^{*}),-(\tilde{v},\tilde{\beta}))\in \disp G((\tilde{x},\tilde{\alpha}),\lambda(\tilde{u},-1))}\langle(\tilde{v}^{*},\tilde{\beta}^{*}), (\tilde{v},\tilde{\beta})\rangle.
\end{array}
\end{align}
Pick $v^{*}\in D^{*}(\partial f)(x,u)(v)$ and find $(\tilde{x},\tilde{u})$, $(\tilde{v}^{*},\tilde{v})$ with $\tilde{v}^{*}\in\widehat{D}^{*}(\partial f)(\tilde{x},\tilde{u})(\tilde{v})$ such that $(\tilde{x},\tilde{u})\rightarrow(x,u)$ with $\tilde{u}\in\partial f(\tilde{x})$ and $(\tilde{v}^{*},\tilde{v})\rightarrow(v^{*},v)$. We get from \eqref{eq:dom-regular-coderivative-dom-graphical-derivative}, \cite[Lemma~3.6]{Chieu2021}, and \cite[Proposition~8.37]{Rockafellar1998} that
\begin{align}\label{eq:inequlity-minimal-coderivative-2}
-d^{2}f(\tilde{x},\tilde{u})(\tilde{v})=\langle -\tilde{z},\tilde{v}\rangle\geq\langle \tilde{v}^{*},-\tilde{v}\rangle\rightarrow-\langle v^{*},v\rangle
\end{align}
whenever $-\tilde{z}\in D(\partial f)(\tilde{x},\tilde{u})(-\tilde{v})$. Denote $\mathcal{T}:=\mathcal{T}_{\operatorname{epi}f}^{2}((\tilde{x},\tilde{\alpha}),(\tilde{v},df(\tilde{x})(\tilde{v})))=\{(w,\gamma)\,|\,d^{2}f(\tilde{x})(\tilde{v},w)\leq\gamma\}$. Since $f$ is parabolically regular, we have the relationships
\begin{align}\label{eq:inequlity-minimal-coderivative-3}
\begin{array}{ll}
&\disp\min_{((\tilde{v}^{*},\tilde{\beta}^{*}),-(\tilde{v},\tilde{\beta}))\in G((\tilde{x},\tilde{\alpha}),\lambda(\tilde{u},-1))}\langle(\tilde{v}^{*},\tilde{\beta}^{*}), (\tilde{v},\tilde{\beta})\rangle\\
&\disp\leq\lambda\left\langle(\nabla h(\tilde{x},\tilde{\alpha})^{T})^{-1}(\tilde{u},-1),\nabla^{2}h(\tilde{x},\tilde{\alpha})((\tilde{v},\tilde{\beta}),(\tilde{v},\tilde{\beta}))\right\rangle\\
&=-\disp\lambda\sup_{(w,\gamma)\in\mathcal{T}}\{\langle w,\tilde{u}\rangle-\gamma\}=\lambda\inf_{w\in\mathcal{R}^{m}}\{d^{2}f(\tilde{x})(\tilde{v},w)-\langle w,\tilde{u}\rangle\}\\
&=\disp\lambda d^{2}f(\tilde{x},\tilde{u})(\tilde{v}),
\end{array}
\end{align}
which allow us to deduce from \eqref{eq:inequlity-minimal-coderivative-1}, \eqref{eq:inequlity-minimal-coderivative-2}, and \eqref{eq:inequlity-minimal-coderivative-3} that
\begin{align}\label{inequality-min-val}
\min_{(z,\gamma)\in D^{*}\mathcal{N}_{\operatorname{epi}f}((x,f(x)),(u,-1))(v,\beta)}\langle(v,\beta), (z,\gamma)\rangle&\leq\liminf_{{\tilde{v}\rightarrow v}\atop{(\tilde{x},\tilde{u}){\xrightarrow{\operatorname{gph}\partial f}}(x,u)}} d^{2}f(\tilde{x},\tilde{u})(\tilde{v})\nonumber\\
&\leq\min_{v^{*}\in D^{*}(\partial f)(x,u)(v)}\langle v^{*},v\rangle.
\end{align}
It follows from \eqref{ineq-Gamma-f-less-cond} that there exist $(\hat{x}^{k},\hat{u}^{k})\rightarrow(x,u)$ with $\hat{u}^{k}\in\partial f(\hat{x}^{k})$ and $\hat{v}^{k}\rightarrow v$ for which
\begin{align}\label{inequality-min-val-geq}
\begin{aligned}
\disp\min_{(z,\gamma)\in D^{*}\mathcal{N}_{\operatorname{epi}f}((x,f(x)),(u,-1))(v,\beta)}\langle(v,\beta),(z,\gamma)\rangle
&\ge\disp\limsup_{k\rightarrow\infty}d^{2}\delta_{\operatorname{epi}f}((\hat{x}^{k},f(\hat{x}^{k})),(\hat{u}^{k},-1))(\hat{v}^{k})\\
&=\disp\limsup_{k\rightarrow\infty}d^{2}f(\hat{x}^{k},\hat{u}^{k})(\hat{v}^{k}),
\end{aligned}
\end{align}
where the last equality is derived from \cite[Lemma~3.6]{Chieu2021}, \cite[Corollary~5.4]{Helmut2019}, and \eqref{eq:inequlity-minimal-coderivative-3}. Therefore, we  conclude from \eqref{ineq-Gamma-f-0}, \eqref{inequality-min-val}, and \eqref{inequality-min-val-geq} that
\begin{align*}
\min_{(z,\gamma)\in D^{*}\mathcal{N}_{\operatorname{epi}f}((x,f(x)),(u,-1))(v,\beta)}\langle(v,\beta), (z,\gamma)\rangle\leq\min_{v^{*}\in D^{*}(\partial f)(x,u)(v)}\langle v^{*},v\rangle\leq\liminf_{k\rightarrow\infty}d^{2}f(\hat{x}^{k},\hat{u}^{k})(\hat{v}^{k})&\\\leq\limsup_{k\rightarrow\infty}d^{2}f(\hat{x}^{k},\hat{u}^{k})(\hat{v}^{k})\leq\min_{(z,\gamma)\in D^{*}\mathcal{N}_{\operatorname{epi}f}((x,f(x)),(u,-1))(v,\beta)}\langle(v,\beta), (z,\gamma)\rangle,&
\end{align*}
which thus completes the proof of the theorem.
\end{proof}\vspace*{-0.2in}

\section{No-Gap Second-Order Conditions for Tilt Stability}\label{sec:no-gap}\vspace*{-0.05in}

In this section, we derive second-order necessary and sufficient and conditions of both neighborhood and pointbased types for tilt-stable local minimizers of the composite optimization problem \eqref{P} under the assumptions above along with the MSCQ for \eqref{P} formulated as follows.

\begin{assumption}\label{assump-3}
Let $x\in\operatorname{dom}g\circ F$, and let $\mathcal{O}$ be a neighborhood of $x$. Suppose that $g$ is locally Lipschitzian around $F(x)$ relative to the domain of $g$ with constant $\ell\in\mathcal{R}_{+}$ and that the mapping $H(x):=\{x'\in\mathcal{O}\,|\,F(x')\in\operatorname{dom}g\}$ satisfies the MSCQ at $x$ for $0$ with modulus $\kappa$, i.e., there exist a constant $\kappa>0$ and a neighborhood $\mathcal{U}$ of $x$ such that
\begin{align*}
\operatorname{dist}(x',\operatorname{dom}g\circ F)\leq\kappa\operatorname{dist}(F(x'),\operatorname{dom}g)\;\mbox{ for all }\;x'\in\mathcal{U}.
\end{align*} 
\end{assumption}

Similarly to the proof of \cite[Proposition~7.1]{Mohammadi2022}, we get the following statement.

\begin{proposition}\label{prop-prox-regular-subdifferential-continuous}
Let $\theta=g\circ F$ with $x\in\operatorname{dom}\th$, and let $g$ satisfy Assumption~{\rm\ref{assump-3}} at $x$. Then $\theta$ is prox-regular and subdifferentially continuous at $x$ for any subgradient $u\in\partial\theta(x)$.  
\end{proposition}	

Denote $\mathcal{G}(x):=\nabla F(x)^{T}\partial g(F(x))$. For a feasible solution $x$ to problem \eqref{P}, pick $x^{*}\in\mathcal{G}(x)$ and then define the {\em multiplier set} at the pair $(x,x^{*})$ by
\begin{align*}
\Lambda(x,x^{*}):=\left\{\mu\in\partial g(F(x))\,\left|\,x^{*}=\nabla F(x)^{T}\mu\right.\right\}
\end{align*}
and the {\em critical cone} at $(x,x^{*})$ by
\begin{align*}
\mathcal{K}(x,x^{*}):=\left\{d\,\left|\,dg(F(x))(\nabla F(x)d)=\langle x^{*},d\rangle\right.\right\}.
\end{align*}
Given any critical direction $v\in\mathcal{K}(x,x^{*})$, the {\em directional multiplier set} is defined by
\begin{align*}
\Lambda(x,x^{*};v):=\mathop{\operatorname{arg\max}}\limits_{\mu\in\Lambda(x,x^{*})}\left(\langle \mu,\nabla^{2}F(x)(v,v)\rangle+d^{2}g(F(x),\mu)(\nabla F(x)v)\right).
\end{align*}

The next lemma plays an important role in the proof of our main results.

\begin{lemma}\label{lemma:value-graphical-derivative} Let $x$ be a feasible solution to problem \eqref{P}, and let $x^{*}\in\mathcal{G}(x)$. Suppose that the function $g$ satisfies Assumption~{\rm\ref{assump-1}} at $x$ for all $\mu\in\Lambda(x,x^{*})$ and Assumption~{\rm\ref{assump-3}} at $x$. Then we have the domain representation
\begin{align}\label{eq:dom-graphical-G}
\operatorname{dom}D\mathcal{G}(x,x^{*})=\mathcal{K}(x,x^{*}).
\end{align}
Moreover, for any $v\in\mathcal{K}(x,x^{*})$ and $w\in D\mathcal{G}(x,x^{*})(v)$, there exists a multiplier $\mu\in\Lambda(x,x^{*};v)\cap(\tau(x,x^{*})\mathbb{B}_{\mathcal{R}^{m}})$ providing the equality
\begin{align*}
\langle v,w\rangle=\langle \mu,\nabla^{2} F(x)(v,v)\rangle+d^{2}g(F(x),\mu)(\nabla F(x)v),
\end{align*}
where $\tau(x,x^{*}):=\kappa\ell\|\nabla F(x)\|+\kappa\|x^{*}\|+\ell$.
\end{lemma}
\begin{proof}
It follows from \cite[Proposition~4.3]{Mohammadi2020} that $\partial (g\circ F)(x)=\mathcal{G}(x)$.  Combining \cite[Proposition~5.1]{Mohammadi2020} and \cite[Lemma~3.6]{Chieu2021}, we obtain \eqref{eq:dom-graphical-G} and the equality
\begin{align*}
\langle v,w\rangle = d^{2}(g\circ F)(x,x^{*})(v).
\end{align*}
Now the claimed results can be deduced from \cite[Theorem~5.4]{Mohammadi2020}.
\end{proof}

The following theorem establishes a {\em neighborhood characterization} of tilt stability in \eqref{P}.
\begin{theorem}\label{tilt-nei} Let $\bar{x}$ be a feasible solution to problem \eqref{P}, and let $f:=f_{0}+g\circ F$. Suppose that the KKT conditions \eqref{eq:kkt-condition} hold at $\bar{x}$, that the function $g$ satisfies Assumption~{\rm \ref{assump-1}} at $F(x)$ for all $\mu\in\partial g(F(x))$ with $x$ in some neighborhood of $\bar{x}$, and  that Assumption~{\rm\ref{assump-3}} is satisfied at $\bar{x}$. Then $\bar{x}$ is a tilt-stable local minimizer for \eqref{P} if and only if that there exists a constant $\eta>0$ such that for all $(x,u)\in\operatorname{gph}\Psi\cap\mathbb{B}_{\eta}(\bar{x},0)$ we have the strict inequality
\begin{align*}
\sup_{\mu\in\Lambda(x,u-\nabla f_{0}(x))\cap(\tau(x,u-\nabla f_{0}(x))\mathbb{B}_{\mathcal{R}^{m}})}\Big\{\langle v,\nabla^{2}_{xx}L(x,\mu)v\rangle+d^{2}g(F(x),\mu)(\nabla F(x)v)\Big\}>0\\
\;\mbox{whenever }\;v\neq 0\;\mbox{ and }\;v\in\mathcal{K}(x,u-\nabla f_{0}(x)),
\end{align*}
where the multifunction $\Psi$ is defined by
\begin{align}\label{Psi}
\Psi(x):=\Big\{\nabla_{x}L(x,\mu)\,\Big|\,\mu\in\partial g(F(x))\Big\}.
\end{align}
\end{theorem}
\begin{proof}
Since MSCQ is fulfilled at $\bar{x}$, it also holds around this point. It follows from \cite[Proposition~1.30]{Mordukhovich2018} and \cite[Theorem~3.6]{Mohammadi2022} that $\partial f(x)=\nabla f_{0}(x)+\mathcal{G}(x)$, and that we have $w-\nabla^{2}f_{0}(x)v\in D\mathcal{G}(x,u-\nabla f_{0}(x))(v)$ for all $v\in\mathcal{K}(x,u-\nabla f_{0}(x))$ and $w\in D(\partial f)(x,u)(v)$.
Employing Lemma~\ref{lemma:value-graphical-derivative} yields the existence of a multiplier $\mu\in\Lambda(x,u-\nabla f_{0}(x);v)\cap(\tau(x,u-\nabla f_{0}(x))\mathbb{B}_{\mathcal{R}^{m}})$ such that
\begin{align*}
\langle w,v\rangle&=\langle\nabla^{2}f_{0}(x)v,v\rangle+\langle\mu,\nabla^{2}F(x)(v,v)\rangle+d^{2}g(F(x),\mu)(\nabla F(x)v)\\
&=\langle v,\nabla^{2}_{xx}L(x,\mu)v\rangle+d^{2}g(F(x),\mu)(\nabla F(x)v).
\end{align*}
Combining the latter with \cite[Theorem~3.3]{Chieu2018} completes the proof of the theorem.
\end{proof}

From now on in this section, we derive {\em pointbased no-gap second-order characterizations} of tilt-stable local minimizers of a given modulus $\kappa>0$ in problem \eqref{P}. By ``no-gap", we understand as usual a pair of conditions such that the sufficient condition is different from the necessary simply by changing the strict inequality sign to its nonstrict counterpart.\vspace*{0.05in}

We start with the following {\em sufficient condition} for tilt stability of local minimizers in \eqref{P} under the {\em similar} general assumptions as in the neighborhood characterizations of Theorem~\ref{tilt-nei}. 
\begin{theorem}\label{theorem-tilt-stability-main}
Let $\bar{x}$ be a feasible solution to problem \eqref{P}. In addition to the assumptions of Theorem~{\rm\ref{tilt-nei}}, suppose that the function $g$ satisfies Assumptions~{\rm\ref{assump-1-1}}, and {\rm\ref{assump-2}} at $F(\ox)$ for all $\mu\in\partial g(F(\ox))$. Given $\kappa>0$, suppose that the condition
\begin{align*}
\langle v,\nabla^{2}_{xx}L(\bar{x},\mu)v\rangle+\Gamma_{g}(F(\bar{x}),\mu)(\nabla F(\bar{x})v)>\frac{1}{\kappa}\|v\|^{2},\quad v\neq 0,\,\nabla F(\bar{x})v\in\bigcup_{V\in{\cal J}\operatorname{Prox}_{g}(F(\bar{x})+\mu)}\operatorname{rge}V
\end{align*}  
holds for all the Lagrange multipliers in the critical directions
\begin{align*}
\mu\in\bigcup_{0\neq w\in\mathcal{K}(\bar{x},-\nabla f_{0}(\bar{x}))}\Lambda(\bar{x},-\nabla f_{0}(\bar{x});w)\cap(\tau(\bar{x},-\nabla f_{0}(\bar{x}))\mathbb{B}_{\mathcal{R}^{m}}). 
\end{align*}
Then $\bar{x}$ is a tilt-stable local minimizer with modulus $\kappa$ for \eqref{P}. 
\end{theorem}
\begin{proof}
Employing Proposition~\ref{prop-prox-regular-subdifferential-continuous} and \cite[Theorem~3.3]{Chieu2018}, it suffices to show that there exists a constant $\eta>0$ for which we have 
\begin{align}\label{eq:positivity-graphical-derivative-f}
\langle w,v\rangle\geq \frac{1}{\kappa}\|v\|^{2}\;\mbox{ whenever }\;w\neq 0,\; w\in D(\partial f)(x,u)(v),\; (x,u)\in\operatorname{gph}\Psi\cap\mathbb{B}_{\eta}(\bar{x},0),
\end{align}
where $\Psi$ is defined in \eqref{Psi}. Suppose on the contrary that \eqref{eq:positivity-graphical-derivative-f} fails and then find sequences $\{x^{k}\}$ and $\{u^{k}\}$ with $x^{k}\rightarrow\bar{x}$, $u^{k}\rightarrow 0$ as $k\rightarrow\infty$, as well as $(v^{k},w^{k})\in\mathcal{T}_{\operatorname{gph}\partial f}(x^{k},u^{k})$ satisfying
\begin{align}\label{eq:inequality-graphical-derivative-contradiction}
\langle w^{k},v^{k}\rangle<\frac{1}{\kappa}\|v^{k}\|^{2}		
\end{align}
for all $k$ sufficiently large. Therefore, $v^{k}\neq 0$ and we may assume that $\|v^{k}\|=1$ for all large $k$. Select a subsequence $\{v^{k_{j}}\}$ with $v^{k_{j}}\rightarrow \bar{v}$ and $\|\bar{v}\|=1$. Since $\partial f(x)=\nabla f_{0}(x)+\mathcal{G}(x)$, deduce from Lemma~\ref{lemma:value-graphical-derivative} that $v^{k_{j}}\in\mathcal{K}(x^{k_{j}},u^{k_{j}}-\nabla f_{0}(x^{k_{j}}))$ and $w^{k_{j}}-\nabla^{2}f_{0}(x^{k_{j}})v^{k_{j}}\in D\mathcal{G}(x^{k_{j}},u^{k_{j}}-\nabla f_{0}(x^{k_{j}}))(v^{k_{j}})$. Then it follows from Lemma~\ref{lemma:value-graphical-derivative} that there exists a sequence of $\mu^{k_{j}}\in\Lambda(x^{k_{j}},u^{k_{j}}-\nabla f_{0}(x^{k_{j}});v^{k_{j}})\cap(\tau(x^{k_{j}},u^{k_{j}}-\nabla f_{0}(x^{k_{j}}))\mathbb{B}_{\mathcal{R}^{m}})$ satisfying the equality
\begin{align*}
\langle w^{k_{j}}-\nabla^{2}f_{0}(x^{k_{j}})v^{k_{j}},v^{k_{j}}\rangle=\langle\mu^{k_{j}},\nabla^{2}F(x^{k_{j}})(v^{k_{j}},v^{k_{j}})\rangle+d^{2}g(F(x^{k_{j}}),\mu^{k_{j}})(\nabla F(x^{k_{j}})v^{k_{j}}),
\end{align*}
which can be equivalently written in the form
\begin{align}\label{eq:equality-graphical-derivative-contradiction}
\langle w^{k_{j}},v^{k_{j}}\rangle=\langle v^{k_{j}},\nabla_{xx}^{2}L(x^{k_{j}},\mu^{k_{j}})v^{k_{j}}\rangle+d^{2}g(F(x^{k_{j}}),\mu^{k_{j}})(\nabla F(x^{k_{j}})v^{k_{j}}).
\end{align}
Since $\{\mu^{k_{j}}\}$ is bounded, we find $\bar{\mu}$ with $\mu^{k_{j}}\rightarrow\bar{\mu}$ as $k_{j}\rightarrow\infty$. It is easy to see that $\bar{\mu}\in\Lambda(\bar{x},-\nabla f_{0}(\bar{x}))\cap(\tau(\bar{x},-\nabla f_{0}(\bar{x}))\mathbb{B}_{\mathcal{R}^{m}})$. Using Assumption~\ref{assump-2}, Lemma~\ref{lemma:subseq-range-regular-coderivative} and the representation
\begin{align*}
    D^{*}\operatorname{Prox}_{g}(F(\bar{x})+\bar{\mu})=\mathop{\operatorname{g-Lim}\sup}\limits_{y\rightarrow F(\ox)+\bar{\mu}}\widehat{D}^{*}\operatorname{Prox}_{g}(y)
\end{align*} 
 via the graphical outer limit (see \cite[Definition~5.32]{Rockafellar1998}) brings us to the inclusion
\begin{align*}
 \bigcup_{V\in{\cal J}\operatorname{Prox}_{g}(F(\bar{x})+\bar{\mu})}\operatorname{rge}V=\operatorname{rge}D^{*}\operatorname{Prox}_{g}(F(\bar{x})+\bar{\mu})\subset\Limsup_{k_{j}\rightarrow\infty}\operatorname{rge}D\operatorname{Prox}_{g}(F(x^{k_{j}})+\mu^{k_{j}}).
\end{align*}
On the other hand, it follows from \cite[Theorem~13.57]{Rockafellar1998} that for any $d\in\operatorname{dom}D(\partial g)(F(x^{k_{j}}),\mu^{k_{j}})$ we have $-d\in\operatorname{dom}D^{*}(\partial g)(F(x^{k_{j}}),\mu^{k_{j}})$ and therefore
\begin{align*}
     \operatorname{dom}D(\partial g)(F(x^{k_{j}}),\mu^{k_{j}})\subset-\operatorname{dom}D^{*}(\partial g)(F(x^{k_{j}}),\mu^{k_{j}}),
\end{align*}
which together with \eqref{eq:inverse-graphical-derivative}, \eqref{eq:coderivative-sum} and 
\begin{align*}
    D^{*}\operatorname{Prox}_{g}(F(\bar{x})+\bar{\mu})
    =\mathop{\operatorname{g-Lim}\sup}\limits_{y\rightarrow F(\ox)+\bar{\mu}}D^{*}\operatorname{Prox}_{g}(y) 
\end{align*}
follows that
\begin{align*}
 \Limsup_{k_{j}\rightarrow\infty}\operatorname{rge}D\operatorname{Prox}_{g}(F(x^{k_{j}})+\mu^{k_{j}})&\subset\Limsup_{k_{j}\rightarrow\infty}\operatorname{rge}D^{*}\operatorname{Prox}_{g}(F(x^{k_{j}})+\mu^{k_{j}})\\
 &\subset \operatorname{rge}D^{*}\operatorname{Prox}_{g}(F(\ox)+\bar{\mu}).
\end{align*}
Combining \eqref{eq:inverse-graphical-derivative}, \eqref{eq:coderivative-sum} and \cite[Lemma~5.2]{MordukhovichTangWang2025} allows us to conclude that \begin{align*}
\Limsup_{k_{j}\rightarrow\infty}\left\{d\,\left|\,dg(F(x^{k_{j}}))(d)=\langle\mu^{k_{j}},d\rangle\right.\right\}&=\Limsup_{k_{j}\rightarrow\infty}\operatorname{rge}D\operatorname{Prox}_{g}(F(x^{k_{j}})+\mu^{k_{j}})\\
&=\operatorname{rge}D^{*}\operatorname{Prox}_{g}(F(\bar{x})+\bar{\mu})=\bigcup_{V\in{\cal J}\operatorname{Prox}_{g}(F(\bar{x})+\bar{\mu})}\operatorname{rge}V,\\
\Limsup_{k_{j}\rightarrow\infty}\mathcal{K}(x^{k_{j}},u^{k_{j}}-\nabla f_{0}(x^{k_{j}}))&=\left\{d\,\left|\,\nabla F(\ox)d\in\operatorname{dom}\Gamma_{g}(F(\bar{x}),\bar{\mu})\right.\right\}.
\end{align*} 	
Thus we get $\nabla F(\bar{x})\bar{v}\in\operatorname{dom}\Gamma_{g}(F(\bar{x}),\bar{\mu})$, and applying \eqref{ineq-Gamma-f-0} tells us that
\begin{align*}
\Gamma_{g}(F(\bar{x}),\bar{\mu})(\nabla F(\bar{x})\bar{v})
&\leq\liminf_{k_{j}\rightarrow\infty} d^{2}g(F(x^{k_{j}}),\mu^{k_{j}})(\nabla F(x^{k_{j}})v^{k_{j}}),
\end{align*}  
which being combined with \eqref{eq:inequality-graphical-derivative-contradiction} and \eqref{eq:equality-graphical-derivative-contradiction} yields
\begin{align}\label{eq:contradict-condition}
\langle\bar{v},\nabla^{2}_{xx}L(\bar{x},\bar{\mu})\bar{v}\rangle+\Gamma_{g}(F(\bar{x}),\bar{\mu})(\nabla F(\bar{x})\bar{v})\leq\frac{1}{\kappa}\|v\|^{2}.
\end{align}
Consider further the following {\em two cases}. First we assume that $x^{k_{j}}\neq\bar{x}$ for {\em infinitely many} $k_{j}$. Let $\frac{x^{k_{j}}-\bar{x}}{\|x^{k_{j}}-\bar{x}\|}\rightarrow w$ by passing to a subsequence if necessary. Since $\mu^{k_{j}}\in\partial g(F(x^{k_{j}}))$ and $\bar{\mu}\in\partial g(F(\bar{x}))$, it follows that  $F(\bar{x})\in\partial g^{*}(\bar{\mu})$ and
\begin{align*}
F(x^{k_{j}})=F(\bar{x})+\nabla F(\bar{x})(x^{k_{j}}-\bar{x})+o(\|x^{k_{j}}-\bar{x}\|)\in\partial g^{*}(\mu^{k_{j}}).
\end{align*}
In this way, we arrive at the relationships
\begin{align}\label{eq:space-nablaF-times-v}
\nabla F(\bar{x})w\in\Limsup_{k^{j}\rightarrow\infty}\frac{\partial g^{*}(\mu^{k_{j}})-F(\bar{x})}{\|x^{k_{j}}-\bar{x}\|}=\Limsup_{k^{j}\rightarrow\infty}\frac{\partial g^{*}(\mu^{k_{j}})-F(\bar{x})}{\|\mu^{k_{j}}-\bar{\mu}\|}\frac{\|\mu^{k_{j}}-\bar{\mu}\|}{\|x^{k_{j}}-\bar{x}\|}.
\end{align}
By Assumption~\ref{assump-1-1}, the subgradient mapping $\partial g$ is calm at $F(\bar{x})$ for $\bar{\mu}$. Thus for $k_{j}$ sufficiently large, there exists a constant $\eta>0$ such that
\begin{align}\label{ineq-calmness}
\frac{\|\mu^{k_{j}}-\bar{\mu}\|}{\|x^{k_{j}}-\bar{x}\|}\leq \eta\|\nabla F(\bar{x})\|+o(1).
\end{align}
Suppose that $\nabla F(\bar{x})\neq 0$. Combining \eqref{eq:space-nablaF-times-v} with \eqref{ineq-calmness} gives us the conditions
\begin{align*}
\nabla F(\bar{x})w&\in\eta\|\nabla F(\bar{x})\|\Limsup_{k^{j}\rightarrow\infty}\frac{\partial g^{*}(\mu^{k_{j}})-F(\bar{x})}{\|\mu^{k_{j}}-\bar{\mu}\|}\\
&=\eta\|\nabla F(\bar{x})\|\Limsup_{k^{j}\rightarrow\infty}\frac{\partial g^{*}(\bar{\mu}+(\mu^{k_{j}}-\bar{\mu}))-F(\bar{x})}{\|\mu^{k_{j}}-\bar{\mu}\|}\\
&\subset\eta\|\nabla F(\bar{x})\|\bigcup_{d\in\mathcal{T}_{\partial g(F(\bar{x}))}(\bar{\mu})}D(\partial g^{*})(\bar{\mu},F(\bar{x}))(d)\\
&=\bigcup_{d\in\mathcal{T}_{\partial g(F(\bar{x}))}(\bar{\mu})}D(\partial g^{*})(\bar{\mu},F(\bar{x}))(d).
\end{align*}
Since $\mathcal{T}_{\partial g(F(\bar{x}))}(\bar{\mu})=\left(\mathcal{N}_{\partial g(F(\bar{x}))}(\bar{\mu})\right)^{\circ}=\left\{d\,\left|\,dg(F(\bar{x}))(d)=\langle\bar{\mu},d\rangle\right.\right\}^{\circ}$, it follows from
Propositions~\ref{lemma:reformulate-dom-df(x,u)-polar} and \ref{proposition-domain-Gamma-f} that for any $d\in\mathcal{T}_{\partial g(F(\bar{x}))}(\bar{\mu})$ with $\nabla F(\bar{x})w\in D(\partial g^{*})(\bar{\mu},F(\bar{x}))(d)$, we have
\begin{align}\label{eq:nablaF-times-v-exact-space}
\begin{aligned}
&0\in D(\partial g^{*})(\bar{\mu},F(\bar{x}))(d),\quad d^{2}g^{*}(\bar{\mu},F(\bar{x}))(d)=d^{2}g(F(\bar{x}),\bar{\mu})(\nabla F(\bar{x})w)=0,\\
&\nabla F(\bar{x})w\in\operatorname{dom}D(\partial g)(F(\bar{x}),\bar{\mu})=\left\{d\,\left|\,dg(F(\bar{x}))(d)=\langle\bar{\mu},d\rangle\right.\right\}.
\end{aligned}
\end{align}
This readily brings us to the inclusion 
\begin{align}\label{eq:value-normal-partial-g}
\lim_{k_{j}\rightarrow\infty}\frac{F(x^{k_{j}})-F(\bar{x})}{\|x^{k_{j}}-\bar{x}\|}\in\left\{d\,\left|\,dg(F(\bar{x}))(d)=\langle\bar{\mu},d\rangle\right.\right\}=\mathcal{N}_{\partial g(F(\bar{x}))}(\bar{\mu})
\end{align} 
which implies in turn that $w\in\mathcal{K}(\bar{x},-\nabla f_{0}(\bar{x}))$. For any $\bar{\mu}\neq\mu\in\Lambda(\bar{x},-\nabla f_{0}(\bar{x}))$, we also get
\begin{equation*}
\nabla F(\bar{x})^{T}\mu=\nabla F(\bar{x})^{T}\bar{\mu}=-\nabla f_{0}(\bar{x}),\ \;\nabla F(\bar{x})w\in\left\{d\,\left|\,dg(F(\bar{x}))(d)=\langle\mu,d\rangle\right.\right\}=\operatorname{dom}d^{2}g(F(\bar{x}),\mu).
\end{equation*}

To show further that $d^{2}g(F(\bar{x}),\tilde{\mu})(\nabla F(\bar{x})w)=0$ for any $\tilde{\mu}=\bar{\mu}+\beta(\mu-\bar{\mu})$  with $\beta\in[0,1]$, deduce from
Proposition~\ref{proposition-domain-Gamma-f} that 
\begin{align*}
&\mu-\bar{\mu}\in\mathcal{T}_{\partial g(F(\bar{x}))}(\bar{\mu}),\quad d^{2}g^{*}(\bar{\mu},F(\bar{x}))(\mu-\bar{\mu})=0,\\ 
&\frac{1}{2}d^{2}g^{*}(\bar{\mu},F(\bar{x}))(\mu-\bar{\mu}) + \frac{1}{2}d^{2}g(F(\bar{x}),\bar{\mu})(\nabla F(\bar{x})w)=\langle\mu-\bar{\mu},\nabla F(\bar{x})w\rangle=0.
\end{align*}
This allows us to conclude by \cite[Theorem~23.5]{Rockafellar1970} that
\begin{align*}
\mu-\bar{\mu}\in D(\partial g)(F(\bar{x}),\bar{\mu})(\nabla F(\bar{x})w)=\Limsup_{{w'\rightarrow\nabla F(\bar{x})w}\atop{t\downarrow 0}}\frac{\partial g(F(\bar{x})+tw')-\bar{\mu}}{t}.
\end{align*}
Therefore, there exists $\beta_{1}>0$ sufficiently small such that
\begin{align*}
&\bar{\mu}+\bar{\beta}_{1}(\mu-\bar{\mu})\in\Limsup_{{w'\rightarrow\nabla F(\bar{x})w}\atop{t\downarrow 0}}\partial g(F(\bar{x})+tw')\;\mbox{ for all }\;\bar{\beta}_{1}\in(0,\beta_{1}).
\end{align*}
More specifically, select $\tilde{\beta}_{1}\in(0,\beta_{1})$ ensuring that
\begin{align*}
(\bar{\mu}+\tilde{\beta}_{1}(\mu-\bar{\mu}))+\tau(\bar{\mu}-\mu)\in\Limsup_{{w'\rightarrow\nabla F(\bar{x})w}\atop{t\downarrow 0}}\partial g(F(\bar{x})+tw')\;\mbox{ whenever }\;\tau\in(0,\tilde{\beta}_{1})
\end{align*} 
and then obtain for all $\tilde{\beta}_{1}\in(0,\beta_{1})$ that
\begin{align*}
\bar{\mu}-\mu\in D(\partial g)(F(\bar{x}),\bar{\mu}+\tilde{\beta}_{1}(\mu-\bar{\mu}))(\nabla F(\bar{x})w)=\Limsup_{{w'\rightarrow\nabla F(\bar{x})w}\atop{t\downarrow 0}}\frac{\partial g(F(\bar{x})+tw')-(\bar{\mu}+\tilde{\beta}_{1}(\mu-\bar{\mu}))}{t}.
\end{align*} 
Combining this with \cite[Lemma~3.6]{Chieu2021} tells us that 
\begin{align*}
d^{2}g(F(\bar{x}),\bar{\mu}+\tilde{\beta}_{1}(\mu-\bar{\mu}))(\nabla F(\bar{x})w)=\langle\bar{\mu}-\mu,\nabla F(\bar{x})w\rangle=0.
\end{align*}
Similarly, we can find small positive numbers $\beta_{2},\ldots,\beta_{k}$ and $\tilde{\beta}_{i}\in(0,\beta_{i})$ for which
\begin{align*}
d^{2}g\left(F(\bar{x}),\bar{\mu}+\sum_{i=1}^{k}\tilde{\beta}_{i}(\mu-\bar{\mu})\right)(\nabla F(\bar{x})w)=\langle\bar{\mu}-\mu,\nabla F(\bar{x})w\rangle=0.
\end{align*}
Letting $i_{0}:=0$ and $\tilde{\beta}_{0}:=0$ and using the finite covering of compact intervals allow us to find a finite integer $K$ and $1=i_{1}<\ldots<i_{K}$ such that 
\begin{align*}
[\tilde{\beta}_{1},1]\subset\bigcup_{t=1}^{K}\left(\sum_{j=0}^{t-1}\tilde{\beta}_{i_{j}},\sum_{j=0}^{t-1}\tilde{\beta}_{i_{j}}+\beta_{i_{t}}\right)\;\mbox{ with }\;\sum\limits_{j=1}^{K}\tilde{\beta}_{i_{j}}=1\;\mbox{ and}
\end{align*}
\begin{align}\label{eq:d2g-Fx-mu}
d^{2}g(F(\bar{x}),\mu)(\nabla F(\bar{x})w)=\langle\bar{\mu}-\mu,\nabla F(\bar{x})w\rangle=0.
\end{align}
By \eqref{eq:nablaF-times-v-exact-space}, \eqref{eq:value-normal-partial-g}, and \eqref{eq:d2g-Fx-mu}, this implies that
\begin{align*}
&\langle\mu,\nabla^{2}F(\bar{x})(w,w)\rangle+d^{2}g(F(\bar{x}),\mu)(\nabla F(\bar{x})w)-\left(\langle\bar{\mu},\nabla^{2}F(\bar{x})(w,w)\rangle+d^{2}g(F(\bar{x}),\bar{\mu})(\nabla F(\bar{x})w)\right)\\
&=\langle\mu-\bar{\mu},\nabla^{2}F(\bar{x})(w,w)\rangle=\lim_{k_{j}\rightarrow\infty}\frac{\left\langle x^{k_{j}}-\bar{x},\nabla^{2}\langle\mu-\bar{\mu},F\rangle(\bar{x})(x^{k_{j}}-\bar{x})\right\rangle}{\|x^{k_{j}}-\bar{x}\|^{2}}\\
&=2\lim_{k_{j}\rightarrow\infty}\frac{\langle\mu-\bar{\mu},F(x^{k_{j}})-F(\bar{x})\rangle}{\|x^{k_{j}}-\bar{x}\|^{2}}\leq 0,
\end{align*}
 which yields $\bar{\mu}\in\Lambda(\bar{x},-\nabla f_{0}(\bar{x});w)\cap(\tau(\bar{x},-\nabla f_{0}(\bar{x}))\mathbb{B}_{\mathcal{R}^{m}})$. When $\nabla F(\bar{x})=0$, the same inclusion relation can be trivially obtained.

 Now we consider the remaining case where there exist only {\em finitely many} $k$ such that $x^{k_{j}}\neq\bar{x}$. In this case, $x^{k_{j}}=\bar{x}$ when $k_{j}$ is sufficiently large. Since $v^{k_{j}}\in\mathcal{K}(\bar{x},u^{k_{j}}-\nabla f_{0}(\bar{x}))$, it follows that $\bar{v}\in\mathcal{K}(\bar{x},-\nabla f_{0}(\bar{x}))$. By Proposition~\ref{proposition-domain-Gamma-f}, the function $g$ is properly twice epi-differentiable at $F(\bar{x})$ for all $\tilde{\mu}\in\partial g(F(\ox))$ with the domain $\operatorname{dom}d^{2}g(F(\bar{x}),\tilde{\mu})$ described as $\left\{d\, \left|\, dg(F(\bar{x})(d)=\langle \tilde{\mu},d\rangle\right.\right\}$.Since the equality $\langle v^{k_{j}},\nabla F(\bar{x})^{T}\mu\rangle=\langle v^{k_{j}},\nabla F(\bar{x})^{T}\bar{\mu}\rangle$ holds for all $\mu\in\Lambda(\bar{x},-\nabla f_{0}(\bar{x}))$, we conclude that
\begin{align}\label{eq:d2g-equality-different-mu}
 \begin{array}{ll}
 d^{2}g(F(\bar{x}),\bar{\mu})(\nabla 
 \disp F(\bar{x})\bar{v})&=\disp\lim_{k_{j}\rightarrow\infty}\disp\frac{g(F(\bar{x})+\tau\nabla F(\bar{x})v^{k_{j}})-g(F(\bar{x}))-\tau\langle \bar{\mu},\nabla F(\bar{x})v^{k_{j}}\rangle}{\frac{1}{2}\tau^{2}}\\
 &=d^{2}g(F(\bar{x}),\mu)(\nabla F(\bar{x})\bar{v})\;\mbox{ whenever }\;\mu\in\Lambda(\bar{x},-\nabla f_{0}(\bar{x})).
\end{array}
\end{align} 
Furthermore, it follows that for all $\mu\in\Lambda(\bar{x},-\nabla f_{0}(\bar{x}))$ we get
 \begin{align}\label{eq:inequality-different-mu}
\begin{array}{ll}
\disp\langle\mu-\bar{\mu},\nabla^{2}F(\bar{x})(\bar{v},\bar{v})\rangle&=\disp 2\lim_{t\downarrow0}\left\langle\mu-\bar{\mu},\frac{F(\bar{x}+t\bar{v})-F(\bar{x})}{t^{2}}\right\rangle\\
 &=\disp2\lim_{t\downarrow0}\lim_{k_{j}\rightarrow\infty}\left\langle\mu-\mu^{k_{j}},\frac{F(\bar{x}+tv^{k_{j}})-F(\bar{x})}{t^{2}}\right\rangle\le 0,
\end{array}
\end{align}
 where the last inequality is deduced from the fact that $\nabla F(\bar{x})v^{k_{j}}\in\mathcal{N}_{\partial
 g(F(\bar{x}))}(\mu^{k_{j}})$. Combining \eqref{eq:d2g-equality-different-mu} and \eqref{eq:inequality-different-mu} leads us to the inclusion $\bar{\mu}\in\Lambda(\bar{x},-\nabla f_{0}(\bar{x});\bar{v})\cap(\tau(\bar{x},-\nabla f_{0}(\bar{x}))\mathbb{B}_{\mathcal{R}^{m}})$. It finally follows from the above discussion and the second-order condition for $\bar{v}$ that 
\begin{align*}
\langle\bar{v},\nabla^{2}_{xx}L(\bar{x},\bar{\mu})\bar{v}\rangle+\Gamma_{g}(F(\bar{x}),\bar{\mu})(\nabla F(\bar{x})\bar{v})>\frac{1}{\kappa}\|\bar{v}\|^{2},
\end{align*}
which contradicts \eqref{eq:contradict-condition} and thus completes the proof of the theorem.
\end{proof}

As we see, Theorem~\ref{theorem-tilt-stability-main} establishes a pointbased second-order sufficient condition for tilt stability in the general composite problem \eqref{P} under the fulfillment of MSCQ. When specialized to the case where $f$ is the indicator function of the (polyhedral) positive orthant, the term $\Gamma_{g}(F(\bar{x}),\mu)(\nabla F(\bar{x})\bar{v})$ {\em vanishes}. In this setting, Theorem~\ref{theorem-tilt-stability-main} reduces to the sufficient condition for tilt stability in nonlinear programming problems derived in  \cite[Theorem~6.1]{GfrererMordukhovich2015}.\vspace*{0.05in}

The pointbased second-order sufficient condition for tilt stability {\em without the modulus specification} is given in the following corollary.

\begin{corollary}
Let $\bar{x}$ be a feasible solution to \eqref{P} with $f=f_{0}+g\circ F$. Suppose that $\ox$ solves the KKT system \eqref{eq:kkt-condition} and that $g$ satisfies the assumptions of Theorem~{\rm\ref{theorem-tilt-stability-main}}. If the condition
\begin{align*}
\langle v,\nabla^{2}_{xx}L(\bar{x},\mu)v\rangle+\Gamma_{g}(F(\bar{x}),\mu)(\nabla F(\bar{x})v)>0,\quad v\neq 0,\,\nabla F(\bar{x})v\in\bigcup_{V\in{\cal J}\operatorname{Prox}_{g}(F(\bar{x})+\mu)}\operatorname{rge}V
\end{align*}  
holds for all the Lagrange multipliers in critical directions
\begin{align*}
\mu\in\bigcup_{0\neq w\in\mathcal{K}(\bar{x},-\nabla f_{0}(\bar{x}))}\Lambda(\bar{x},-\nabla f_{0}(\bar{x});w)\cap(\tau(\bar{x},-\nabla f_{0}(\bar{x}))\mathbb{B}_{\mathcal{R}^{m}}),
\end{align*}
then $\bar{x}$ is a tilt-stable local minimizer for problem \eqref{P}.
\end{corollary}

%To derive the no-gap pointbased second-order {\em necessary condition} for tilt-stable local minimizers of problem \eqref{P}, we first introduce the following new notion.

%\begin{definition}
%Let $\ox$ be a feasible solution to problem \eqref{P}. We say that $\ox$ {\sc nondegenerates in the critical direction} $v\in\mathcal{R}^{m}$ if the set of Lagrange multipliers $\Lambda(\ox,-\nabla f_{0}(\ox);v)$ at $(\ox,-\nabla f_{0}(\ox))$ in the direction $v$ is a singleton.
%\end{definition}

Here is the aforementioned pointbased necessary condition for tilt stability in \eqref{P}.

\begin{theorem}\label{thm-main-necessary-cond-tilt-stability}
Let $\bar{x}$ be a tilt-stable local minimizer with modulus $\kappa>0$ for problem \eqref{P}. In addition to the assumptions of Theorem~{\rm\ref{tilt-nei}},  and Assumption~{\rm\ref{assump-2}} at $F(\ox)$ for all $\mu\in\partial g(F(\ox))$, suppose that for any nonzero critical direction 
$w\in\mathcal{K}(\ox,-\nabla f_{0}(\ox))$, $\mu\in\Lambda(\ox,-\nabla f_{0}(\ox);w)\cap(\tau(\bar{x},-\nabla f_{0}(\bar{x}))\mathbb{B}_{\mathcal{R}^{m}})$ and $\nabla F(\ox)v\in\operatorname{dom}\Gamma_{g}(F(\ox),\mu)$, and there exist $(x^{k},\mu^{k})\rightarrow(\ox,\mu)$ with $\mu^{k}\in\partial g(F(x^{k}))$ and $z^{k}\rightarrow(\overline{W}^{\dagger}-I)\nabla F(\ox)v$, $v^{k}\rightarrow v$ with $(z^{k},-\nabla F(x^{k})v^{k})\in\widehat{\mathcal{N}}_{\operatorname{gph}\partial g}(F(x^{k}),\mu^{k})$ satisfying that $\Lambda(x^{k},\nabla F(x^{k})^{T}\mu^{k})$ is a singleton.
Then we have the pointbased second-order necessary condition for tilt stability
\begin{align*}
\langle v,\nabla^{2}_{xx}L(\bar{x},\mu)v\rangle+\Gamma_{g}(F(\bar{x}),\mu)(\nabla F(\bar{x})v)\geq\frac{1}{\kappa}\|v\|^{2},\quad v\neq 0,\;\nabla F(\bar{x})v\in\bigcup_{V\in{\cal J}\operatorname{Prox}_{g}(F(\bar{x})+\mu)}\operatorname{rge}V
\end{align*}  
with modulus $\kappa$ that holds for all the Lagrange multipliers in critical directions
\begin{align*}
\mu\in\bigcup_{0\neq w\in\mathcal{K}(\bar{x},-\nabla f_{0}(\bar{x}))}\Lambda(\bar{x},-\nabla f_{0}(\bar{x});w)\cap(\tau(\bar{x},-\nabla f_{0}(\bar{x}))\mathbb{B}_{\mathcal{R}^{m}}).
\end{align*}
\end{theorem}
\begin{proof}
Suppose that $\bar{x}$ is a tilt-stable local minimizer with modulus $\kappa$ while, on the contrary, there exist vectors $w\in\mathcal{K}(\bar{x},-\nabla f_{0}(\bar{x}))$, $\mu\in\Lambda(\bar{x},-\nabla f_{0}(\bar{x});w)\cap(\tau(\bar{x},-\nabla f_{0}(\bar{x}))\mathbb{B}_{\mathcal{R}^{m}})$, and 
\begin{align*}
v\neq 0,\;\nabla F(\bar{x})v\in\bigcup_{V\in{\cal J}\operatorname{Prox}_{g}(F(\bar{x})+\mu)}\operatorname{rge}V
\end{align*}
satisfying the opposite strict inequality
\begin{align}\label{eq:necessary-cond-contradict-assump}
\langle v,\nabla^{2}_{xx}L(\bar{x},\mu)v\rangle+\Gamma_{g}(F(\bar{x}),\mu)(\nabla F(\bar{x})v)<\frac{1}{\kappa}\|v\|^{2}.
\end{align}
Therefore, there exist sequences $(x^{k},\mu^{k})\rightarrow(\ox,\mu)$ with $\mu^{k}\in\partial g(F(x^{k}))$ and $(z^{k},v^{k})\rightarrow((\overline{W}^{\dagger}-I)\nabla F(\ox)v,v)$ with $z^{k}\in\widehat{D}^{*}(\partial g)(F(x^{k}),\mu^{k})(\nabla F(x^{k})v^{k})$ satisfying $\Lambda(x^{k},\nabla F(x^{k})^{T}\mu^{k})=\{\mu^{k}\}$. Applying \cite[Corollary~3.7]{Mohammadi2022}, we obtain $\Lambda(x^{k},\nabla F(x^{k})^{T}\mu^{k};v^{k})\cap(\tau(\bar{x},-\nabla f_{0}(\bar{x}))\mathbb{B}_{\mathcal{R}^{m}})=\{\mu^{k}\}$. It follows from \eqref{eq:dom-regular-coderivative-dom-graphical-derivative} and \cite[Lemma~5.2]{MordukhovichTangWang2025} that $-\nabla F(x^{k})v^{k}\in\operatorname{dom}d^{2}g(F(x^{k}),\mu^{k})$.
Employing Lemma~\ref{lemma:value-graphical-derivative} and \cite[Theorem~3.3]{Chieu2018} brings us to the inequalities
\begin{align*}
\langle v^{k},\nabla^{2}_{xx}L(x^{k},\mu^{k})v^{k}\rangle+d^{2}g(F(x^{k}),\mu^{k})(-\nabla F(x^{k})v^{k})\geq\frac{1}{\kappa}\|v^{k}\|^{2},\quad k\in\mathbb N.
\end{align*}
Using Assumption~\ref{assump-2}, Proposition~\ref{proposition-reformulate-Gamma-f} and (i) of Theorem~\ref{proposition-inequality-Gamma-f} along $k\to\infty$, we arrive at a contradiction with the assumption in \eqref{eq:necessary-cond-contradict-assump} and thus complete the proof of the theorem. 
\end{proof}\vspace*{-0.2in}

\section{Conclusions and Further Research}\label{conc}
\label{sec:Conclusion}\vspace*{-0.05in}

The main achievements of this paper provide no-gap pointbased characterizations of tilt-stable local minimizers for a broad class of composite optimization problems without nondegeneracy assumptions. Just a few results have been obtained in this direction for problems of nonlinear programming and---partly---for second-order cone programming with a single Lorentz cone. Based on advanced tools of variational analysis and generalized differentiation, we develop here a general scheme to establish such pointbased no-gap characterizations of tilt stability without nondegeneracy under fairly mild assumptions that include the weakest metric subregularity constraint qualification. The obtained results involve a novel notion of the second-order variational function whose generalized differential properties play a crucial role in deriving our powerful characterizations of tilt-stable minimizers for nonpolyhedral and possible degenerate programs. Along with the conventional assumptions on the program data,  we also impose a technical requirement (Assumption~\ref{assump-2}) that is inherent in the class of composite optimization problems under consideration. This assumption is constructively verified in the appendix below for the spectral norm function in the composition, which is highly important for various applications. 

 Among topics of our future research, we mention revealing broader classes of composite optimization problems, where Assumption~\ref{assump-2} and the pointbased no-gap characterizations of tilt stability in Theorems~\ref{theorem-tilt-stability-main} and \ref{thm-main-necessary-cond-tilt-stability} can be explicitly calculated in terms of the given data. Moreover, we intend to incorporate the established no-gap pointbased characterizations of tilt stability into the design and justification of numerical methods of optimization; compare, e.g., \cite{Khanh2023,Mordukhovich2024,Mordukhovich2021} for particular classes of problems in constrained optimization.\vspace*{-0.12in}

\begin{appendices}
\section{Assumption Verification for the Spectral Norm Function} \label{appendices}\vspace*{-0.05in}

Let $p,q\in\mathbb N$ with $p\leq q$, and let $\|\cdot\|_{2}$ be the {\em matrix spectral norm function}, which is defined as the largest singular value of the matrix in question. We show below that Assumption~\ref{assump-2} holds where $g$ is the spectral norm function. Given $X\in\mathcal{R}^{p\times q}$ and $U\in\partial \|X\|_{2}$, let $A:=X+U$ and suppose that the singular value decomposition of $A$ takes the form
$A=R[\Sigma\ 0]S^{T}$, where $R\in\mathcal{R}^{p\times p}$ and $S\in\mathcal{R}^{q\times q}$ are orthogonal matrices, $\Sigma=\operatorname{Diag}(\sigma)$ is the diagonal matrix of singular values of matrix $A$ with $\sigma=(\sigma_{1},\ldots,\sigma_{p})$, and $\sigma_{1}\geq\ldots\geq\sigma_{r}>0=\sigma_{r+1}=\ldots=
\sigma_{p}$.
		
Denote $\mathrm{B}:=\{x\in\mathcal{R}^{p}\,|\,\|x\|_{1}\leq 1\}$ and $\mathcal{B}:=\{X\in\mathcal{R}^{p\times q}\,|\,\|X\|_{*}\leq 1\}$, where $\|\cdot\|_{*}$ is the {\em nuclear norm function}, i.e., the sum of all the singular values of a given matrix. It follows from \cite[Proposition~2.2]{ChenLiuSunToh2016} that the proximal mapping of $\|\cdot\|_{2}$ is represented as
\begin{align*}
\operatorname{Prox}_{\|\cdot\|_{2}}(A)=A-\Pi_{\mathcal{B}}(A),
\end{align*}
where $\Pi_{\mathcal{B}}$ is the projection onto $\mathcal{B}$ given by
$\Pi_{\mathcal{B}}(A)=R[\operatorname{Diag}(\Pi_{\mathrm{B}}(\sigma))\ 0]S^{T}$.
Let further $\mathcal{O}_{F}$ be the collection of points where $\Pi_{\mathcal{B}}$ is differentiable, let $\mathcal{E}:=\{A\in\mathcal{O}_{F}\,|\,\sigma_{1}(A)>\sigma_{2}(A)>\ldots>\sigma_{p}(A)>0\}$, and let $k_{1}(\sigma)$ and $k_{2}(\sigma)$ be the maximal indexes of the sets
\begin{align*}
\left\{i\,\left|\,\sigma_{i}>\frac{1}{i}\left(\sum_{j=1}^{i}\sigma_{j}-1\right)\right.\right\},\quad\left\{i\,\left|\,\sigma_{i}\geq\frac{1}{i}\left(\sum_{j=1}^{i}\sigma_{j}-1\right)\right.\right\},
\end{align*}
respectively. Taking a sequence $\{A^{i}\}$ converging to $A$ with all $A^{i}\in\mathcal{E}$, suppose that the singular value decomposition of $A^{i}$ is $A^{i}=R^{i}[\operatorname{Diag}(\sigma^{i})\ 0](S^{i})^{T}$, where $R^{i}\in\mathcal{R}^{p\times p}$ and $S^{i}\in\mathcal{R}^{q\times q}$ are orthogonal matrices, $\sigma^{i}=(\sigma_{1}^{i},\ldots,\sigma_{p}^{i})$ is the singular value vector of $A^{i}$ and $\sigma^{i}_{1}>\sigma^{i}_{2}>\ldots>\sigma^{i}_{p}>0$. Letting $s^{i}\in[k_{1}(\sigma^{i}),k_{2}(\sigma^{i})]$ and $s\in[k_{1}(\sigma),k_{2}(\sigma)]$, denote  
\begin{align*}
\Pi_{\mathrm{B}}(\sigma^{i}):=(p_{1}^{i},\ldots,p_{s^{i}}^{i},0,\ldots,0)^{T},\quad\Pi_{\mathrm{B}}(\sigma):=(p_{1},\ldots,p_{s},0,\ldots,0)^{T}.
\end{align*}
Based on \cite[Proposition~4.1]{ChenLiuSunToh2016}, we verify the claimed result in the following three cases.\\[0.5ex]
{\bf Case~1}: $\|A\|_{*}<1$. In this case, ${\cal J}\operatorname{Prox}_{\|\cdot\|_{2}}(A)=\{0\}$ is a singleton, and the fulfillment of Assumption~\ref{assump-2} is obvious.\\[0.5ex]	
{\bf Case~2}: $\|A\|_{*}=1$. Since $\Pi_{\mathcal{B}}$ is differentiable on the relative interior of $\mathcal{B}$, the origin is always an element of ${\cal J}\operatorname{Prox}_{\|\cdot\|_{2}}(A)$. Without loss of generality, it is sufficient to consider the case where $\|A^{i}\|_{*}>1$. 
Observing that $k_{1}(\sigma)=r$ and $k_{2}(\sigma)=p$ in this case, denote 
\begin{align*}
\alpha_{1}:=\{1,\ldots,r\},\quad\alpha_{2}:=\{r+1,\ldots,s\},\quad
\alpha_{3}:=\{s+1,\ldots,p\},\quad\alpha_{4}:=\{p+1,\ldots,q\}
\end{align*} 
and then define $\Delta^{i}\in\mathcal{R}^{(s-r)\times(p-s)}$, $\Theta^{i}\in\mathcal{R}^{(s-r)\times(s-r)}$, $\Lambda^{i}\in\mathcal{R}^{(s-r)\times(p-s)}$, and $\Upsilon^{i}\in\mathcal{R}^{s-r}$ by
\begin{align*}
\Delta^{i}_{kj}&=\frac{p_{k+r}^{i}}{\sigma^{i}_{k+r}-\sigma^{i}_{j+s}}\;\forall\;k\in\alpha_{2}-r,\,j\in\alpha_{3}-s,\quad
\Theta^{i}_{kj}=\frac{p^{i}_{k+r}+p^{i}_{j+r}}{\sigma^{i}_{k+r}+\sigma^{i}_{j+r}}\;\forall\;k\in\alpha_{2}-r,\,j\in\alpha_{2}-r,\\
\Lambda^{i}_{kj}&=\frac{p^{i}_{k+r}}{\sigma^{i}_{k+r}+\sigma^{i}_{j+s}},\;\forall\;k\in\alpha_{2}-r,\,j\in\alpha_{3}-s,\quad
\Upsilon^{i}_{k}=\frac{p^{i}_{k+r}}{\sigma^{i}_{k+r}}\;\forall\;k\in\alpha_{2}-r.
\end{align*}
Let $\mathcal{S}_{\Delta}$, $\mathcal{S}_{\Theta}$, $\mathcal{S}_{\Lambda}$, and $\mathcal{S}_{\Upsilon}$ be the collections of accumulation points of the sequences $\{\Delta^{i}\}$, $\{\Theta^{i}\}$, $\{\Lambda^{i}\}$, and $\{\Upsilon^{i}\}$, respectively. Picking
$\mathcal{V}\in\partial_{B}\operatorname{Prox}_{\|\cdot\|_{2}}(A)$ and $D\in\mathcal{R}^{p\times q}$, we get that $I-\mathcal{V}\in\partial_{B}\Pi_{\mathcal{B}}(A)$. It follows from \cite[Proposition~4.1]{ChenLiuSunToh2016} that
\begin{align*}
(I-\mathcal{V})D=\overline{R}[Z_{1}(s)\ Z_{2}(s)]\overline{S}^{T},
\end{align*}
where $\overline{R}\in\mathcal{R}^{p\times p}$ and $\overline{S}\in\mathcal{R}^{q\times q}$ are singular vector matrices of $A$, where $\overline{S}=[\overline{S}_{1}\ \overline{S}_{2}]$ with $\overline{S}_{1}\in\mathcal{R}^{q\times p}$ and $\overline{S}_{2}\in\mathcal{R}^{q\times(q-p)}$, where $Z_{1}(s)\in\mathcal{R}^{p\times p}$ and $Z_{2}(s)\in\mathcal{R}^{p\times(q-p)}$ are defined by
\begin{align*}
Z_{1}(s)=\Omega(s)\circ D_{1}^{s}+\Gamma(s)\circ D_{1}^{a}-\frac{\operatorname{Tr}(D_{11}(s))}{s}\left[\begin{array}{cc}
I_{s} & 0\\ 0 & 0
\end{array}\right],\quad
Z_{2}(s)=\left[\begin{array}{c}
1_{r\times(q-p)}\\ \Upsilon_{\alpha_{2}}1^{T}_{q-p} \\0
\end{array}\right]\circ D_{2}
\end{align*}
with $\Upsilon_{\alpha_{2}}\in\mathcal{S}_{\Upsilon}$, $D_{1}^{s}:=\frac{D_{1}+D_{1}^{T}}{2}$, $D_{1}^{a}:=\frac{D_{1}-D_{1}^{T}}{2}$, $D_{1}:=\overline{R}^{T}D\overline{S}_{1}$, $D_{2}:=\overline{R}^{T}D\overline{S}_{2}$, where $D_{11}(s)$ is the matrix extracted from the first $s$ rows and columns of $D_{1}$, and where $\Omega(s)$ and $\Gamma(s)$ are
\begin{align*}
\Omega(s):=\left[\begin{array}{ccc}
1_{r\times r} & 1_{r\times (s-r)} & 1_{r\times(p-s)}\\
1_{(s-r)\times r} & 1_{(s-r)\times(s-r)} & \Omega_{\alpha_{2}\alpha_{3}}\\
1_{(p-s)\times r} & \Omega_{\alpha_{2}\alpha_{3}}^{T} & 0
\end{array}\right],\quad
\Gamma(s):=\left[\begin{array}{ccc}
1_{r\times r} & 1_{r\times (s-r)} & 1_{r\times(p-s)}\\
1_{(s-r)\times r} & \Gamma_{\alpha_{2}\alpha_{2}} & \Gamma_{\alpha_{2}\alpha_{3}}\\
1_{(p-s)\times r} & \Gamma_{\alpha_{2}\alpha_{3}}^{T} & 0
\end{array}\right]
\end{align*}
with $\Omega_{\alpha_{2}\alpha_{3}}\in\mathcal{S}_{\Delta}$, $\Gamma_{\alpha_{2}\alpha_{2}}\in\mathcal{S}_{\Theta}$, and $\Gamma_{\alpha_{2}\alpha_{3}}\in\mathcal{S}_{\Lambda}$. 
For any $\mathcal{W}\in{\cal J}\operatorname{Prox}_{\|\cdot\|_{2}}(A)$, $D\in\mathcal{R}^{p\times q}$, there are $\mathcal{V}^{j}\in\partial_{B}\operatorname{Prox}_{\|\cdot\|_{2}}(A)$ and $\lambda_{j}\in[0,1]$ as $j=1,\ldots,t$ with $\sum\limits_{j=1}^{t}\lambda_{j}=1$ satisfying $\mathcal{W}=\sum\limits_{j=1}^{t}\lambda_{j}\mathcal{V}^{j}$, and there exist $s_{1},\ldots,s_{t}\in[k_{1}(\sigma),k_{2}(\sigma)]$, $s_{1}\leq s_{2}\leq\ldots\leq s_{t}$ such that
\begin{align*}
\mathcal{W}D=D-\overline{R}[E_{1}\ E_{2}]\overline{S}^{T},
\end{align*}
where $E_{1}\in\mathcal{R}^{p\times p}$, $E_{2}\in\mathcal{R}^{p\times(q-p)}$ are defined by
\begin{align*}
E_{1}:=\widehat{\Omega}\circ D_{1}^{s}+\widehat{\Gamma}\circ D_{1}^{a}-\left[\begin{array}{cc}
\widehat{\Lambda} & 0 \\
0 & 0
\end{array}\right]\ \mbox{and}\ E_{2}:=\left[\begin{array}{c}\widehat{\Upsilon}1^{T}_{q-p}\\ 0\end{array}\right]\circ D_{2}
\end{align*}
with the notations
\begin{align*}
\widehat{\Omega}&:=\sum_{i=1}^{t}\lambda_{i}\Omega(s_{i}),\quad \widehat{\Gamma}:=\sum_{i=1}^{t}\lambda_{i}\Gamma(s_{i}),\quad \widehat{\Upsilon}:=\sum_{i=1}^{t}\left[\begin{array}{c}1_{r}\\ \lambda_{i}\Gamma_{\alpha_{2}}\end{array}\right],\\
\widehat{\Lambda}&:=\operatorname{Diag}\left\{\underbrace{\sum_{l=1}^{t}\frac{\lambda_{l}\operatorname{Tr}D_{11}(s_{l})}{s_{l}},\dots,\sum_{l=1}^{t}\frac{\lambda_{l}\operatorname{Tr}D_{11}(s_{l})}{s_{l}}}_{r},\sum_{{l=1}\atop{s_{l}\geq r+1}}^{t}\frac{\lambda_{l}\operatorname{Tr}D_{11}(s_{l})}{s_{l}},\ldots,\sum_{{l=1}\atop{s_{l}\geq s_{t}}}^{t}\frac{\lambda_{l}\operatorname{Tr}D_{11}(s_{l})}{s_{l}}\right\}.
\end{align*} 
Letting $s=r$, we get $\alpha_{2}=\emptyset$ and find
$\overline{\mathcal{W}}\in\partial_{B}\operatorname{Prox}_{\|\cdot\|_{2}}(A)$ such that
\begin{align*}
\overline{\mathcal{W}}D=D-\overline{R}[Z_{1}(r)\ Z_{2}(r)]\overline{S}^{T}=\overline{R}\left[\begin{array}{cc}
\frac{\operatorname{Tr}(D_{11}(r))}{r}I_{r} & 0_{r\times (q-r)}\\
0_{(p-r)\times r} & \widetilde{D}_{\alpha_{3}\bar{\gamma}}
\end{array}\right]\overline{S}^{T},\ \widetilde{D}=\overline{R}^{T}D\overline{S},\ \bar{\gamma}=\alpha_{3}\cup\alpha_{4}
\end{align*}
for any $D\in\mathcal{R}^{p\times q}$. Since $A=R[\Sigma\ 0]S^{T}=\overline{R}[\Sigma\ 0]\overline{S}^{T}$, it follows from \cite[Proposition~2.4]{ChenLiuSunToh2016} that there exist orthogonal matrices $Q\in\mathcal{R}^{r\times r}$, $Q'\in\mathcal{R}^{(p-r)\times(p-r)}$, and $Q''\in\mathcal{R}^{(q-r)\times(q-r)}$ such that
\begin{align*}
\overline{R}=R\left[\begin{array}{cc}
Q & 0\\ 0 & Q'
\end{array}\right]\;\mbox{ and }\;\overline{S}=S\left[\begin{array}{cc}
Q & 0\\ 0 & Q''
\end{array}\right].
\end{align*}
Therefore, we arrive at the representations
\begin{align*}
\overline{\mathcal{W}}D=R\left[\begin{array}{cc}
\frac{\operatorname{Tr}(D_{11}(r)}{r}I_{r} & 0_{r\times (q-r)}\\
0_{(p-r)\times r} & \widehat{D}_{\alpha_{3}\bar{\gamma}}
\end{array}\right]S^{T}\;\mbox{ and }\;\widehat{D}=R^{T}DS,
\end{align*}
which brings us to the equalities
\begin{align*}
\bigcup_{\mathcal{W}\in{\cal J}\operatorname{Prox}_{\|\cdot\|_{2}}(A)}\operatorname{rge}\mathcal{W}=\operatorname{rge}\overline{\mathcal{W}}=\left\{\left.R\left[\begin{array}{cc}
aI_{r} & 0_{r\times (q-r)}\\
0_{(p-r)\times r} & \widehat{D}
\end{array}\right]S^{T}\,\right|\,a\in\mathcal{R},\,\widehat{D}\in\mathcal{R}^{(p-r)\times(q-r)}\right\}.
\end{align*} 
Furthermore, for any $Y\in\bigcup\limits_{\mathcal{W}\in{\cal J}\operatorname{Prox}_{\|\cdot\|_{2}}(A)}\operatorname{rge}\mathcal{W}$, there exists $\overline{D}\in\mathcal{R}^{p\times q}$ such that $Y=\overline{W}\overline{D}$ and 
\begin{align*}
\overline{D}-Y=R\left[\begin{array}{cc}
\overline{D}_{\alpha_{1} \alpha_{1}}-\frac{\operatorname{Tr}\left(\overline{D}_{\alpha_{1}\alpha_{1}}\right)}{r}I_{r}& \overline{D}_{\alpha_{1}\bar{\gamma} }\\
\overline{D}_{\alpha_{3}\alpha_{1}} & 0_{(p-r)\times(q-r)}
\end{array}\right]S^{T}.
\end{align*} 
Using \cite[Proposition~3.2]{Tangwang2025} tells us that $\langle Y,D-Y\rangle\geq 0$. Therefore,
\begin{align*}
\min_{{D\in\mathcal{R}^{p\times q},Y=\mathcal{W}D}\atop{\mathcal{W}\in{\cal J}\operatorname{Prox}_{\|\cdot\|_{2}}(A)}}\langle Y,D-Y\rangle=\langle Y,\overline{D}-Y\rangle=\sum_{i=1}^{r}\left(\overline{D}_{ii}-\frac{\operatorname{Tr}(\overline{D}_{\alpha_{1}\alpha_{1}})}{r}\right)\frac{\operatorname{Tr}(\overline{D}_{\alpha_{1}\alpha_{1}})}{r}=0,
\end{align*}
which justifies the fulfillment of the Assumption~\ref{assump-2} in this case.\\[0.5ex]
{\bf Case~3}: $\|A\|_{*}>1$. It is easy to check that now we have $k_{1}(\sigma)\leq k_{2}(\sigma)\leq r$. Denote
\begin{align*}
&\beta_{1}:=\{1,\ldots,k_{1}(\sigma)\},\quad\beta_{2}:=\{k_{1}(\sigma)+1,\ldots,s\},\quad
\beta_{3}:=\{s+1,\ldots,k_{2}(\sigma)\},\\
&\beta_{4}:=\{k_{2}(\sigma)+1,\ldots,r\},\quad
\beta_{5}:=\{r+1,\ldots,p\}
\end{align*}
and then define $\Xi^{i}\in\mathcal{R}^{(s-k_{1}(\sigma))\times(k_{2}(\sigma)-s)}$ with the components
\begin{align*}
\Xi^{i}_{kj}:=\frac{p^{i}_{k+k_{1}(\sigma)}}{\sigma^{i}_{k+k_{1}(\sigma)}-\sigma^{i}_{j+s}}\;\mbox{ for all }\;k\in\beta_{2}-k_{1}(\sigma)\;\mbox{ and }\;j\in\beta_{3}-s.
\end{align*}
Let $\mathcal{S}_{\Xi}$ be the collection of accumulation points of the sequence $\{\Xi^{i}\}$. For any $\mathcal{V}\in\partial_{B}\operatorname{Prox}{\|\cdot\|_{2}}(A)$ and $D\in\mathcal{R}^{p\times q}$, deduce from \cite[Proposition~4.1]{ChenLiuSunToh2016} that
\begin{align*}
(I-\mathcal{V})D=\overline{R}[Z_{1}(s)\ Z_{2}(s)]\overline{S}^{T},
\end{align*}
where $\overline{R}\in\mathcal{R}^{p\times p}$ and $\overline{S}\in\mathcal{R}^{q\times q}$ are singular vector matrices of $A$, where $\overline{S}=[\overline{S}_{1}\ \overline{S}_{2}$] with $\overline{S}_{1}\in\mathcal{R}^{q\times p}$ and $\overline{S}_{2}\in\mathcal{R}^{q\times(q-p)}$, and where $Z_{1}(s)\in\mathcal{R}^{p\times p}$ and $Z_{2}(s)\in\mathcal{R}^{p\times(q-p)}$ are defined by
\begin{align*}
Z_{1}(s):=\Omega(s)\circ D_{1}^{s}+\Gamma(s)\circ D_{1}^{a}-\frac{\operatorname{Tr}(D_{11}(s))}{s}\left[\begin{array}{cc}
I_{s} & 0\\ 0 & 0
\end{array}\right]\;\mbox{ and }\;
Z_{2}(s):=\left[\begin{array}{c}
\Upsilon(s)\, 1^{T}_{q-p} \\0
\end{array}\right]\circ D_{2}
\end{align*}
with $D_{1}^{s}:=\frac{D_{1}+D_{1}^{T}}{2}$, $D_{1}^{a}:=\frac{D_{1}-D_{1}^{T}}{2}$, $D_{1}:=\overline{R}^{T}D\overline{S}_{1}$, $D_{2}:=\overline{R}^{T}D\overline{S}_{2}$, where $\Upsilon(s)\in\mathcal{R}^{s}$, and where $D_{11}(s)$ is the matrix extracted from the first $s$ rows and columns of $D_{1}$. The matrices $\Omega(s),\Gamma(s)\in\mathcal{R}^{p\times p}$ and vector $\Upsilon\in\mathcal{R}^{s}$ above are given, respectively, by
\begin{align*}
\Omega(s)&:=\left[\begin{array}{ccc}
1_{s\times s} & \Omega_{\gamma_{1}\gamma_{2}} & \Omega_{\gamma_{1}\beta_{5}} \\ \Omega_{\gamma_{1}\gamma_{2}}^{T} & 0_{(r-s)\times(r-s)} & 0_{(r-s)\times(p-r)} \\
\Omega_{\gamma_{1}\beta_{5}}^{T} & 0_{(p-r)\times(r-s)} & 0_{(p-r)\times(p-r)}
\end{array}\right],\ \gamma_{1}=\beta_{1}\cup\beta_{2},\ \gamma_{2}=\beta_{3}\cup\beta_{4},\\
\Gamma(s)&:=\left[\begin{array}{ccc}
\Gamma_{\gamma_{1}\gamma_{1}} & \Gamma_{\gamma_{1}\gamma_{2}} & \Gamma_{\gamma_{1}\beta_{5}}\\
\Gamma_{\gamma_{1}\gamma_{2}}^{T} & 0_{(r-s)\times(r-s)} & 0_{(r-s)\times(p-s)}\\
\Gamma_{\gamma_{1}\beta_{5}}^{T} & 0_{(p-r)\times(r-s)} & 0_{(p-s)\times(p-s)}
\end{array}\right],
\end{align*}
where $\Omega_{\beta_{2}\beta_{3}}\in\mathcal{S}_{\Xi}$ with the components
\begin{align*}
(\Omega_{\beta_{1}\beta_{3}})_{kj}&:=\frac{p_{k}}{\sigma_{k}-\sigma_{j+s}},\;\forall\;k\in\beta_{1},\;j\in\beta_{3}-s,\\
(\Omega_{\gamma_{1}\beta_{4}})_{kj}&:=\frac{p_{k}}{\sigma_{k}-\sigma_{j+k_{2}(\sigma)}},\;\forall\;k\in\gamma_{1},\;j\in\beta_{4}-k_{2}(\sigma),\quad
(\Omega_{\gamma_{1}\beta_{5}})_{kj}:=\frac{p_{k}}{\sigma_{k}},\;\forall\;k\in\gamma_{1},\;j\in\beta_{5}-r,\\
(\Gamma_{\gamma_{1}\gamma_{1}})_{kj}&:=\frac{p_{k}+p_{j}}{\sigma_{k}+\sigma_{j}},\;\forall\;k\in\gamma_{1},\;j\in\gamma_{1},\quad		
(\Gamma_{\gamma_{1}\gamma_{2}})_{kj}:=\frac{p_{k}}{\sigma_{k}+\sigma_{j+s}},\;\forall\;k\in\gamma_{1},\;j\in\gamma_{2}-s,\\
(\Gamma_{\gamma_{1}\beta_{5}})_{kj}&:=\frac{p_{k}}{\sigma_{k}},\;\forall\;k\in\gamma_{1},\,j\in\beta_{5}-r,\quad
\Upsilon(s)_{k}:=\frac{p_{k}}{\sigma_{k}},\;\forall\;k=1,\ldots,s.
\end{align*}
Since $A=R[\Sigma\ 0]S^{T}=\overline{R}[\Sigma\ 0]\overline{S}^{T}$, it follows from
\cite[Proposition~2.4]{ChenLiuSunToh2016} that there exist orthogonal matrices $Q_{1}\in\mathcal{R}^{s\times s}$, $Q_{2}\in\mathcal{R}^{(r-s)\times(r-s)}$, $Q_{3}\in\mathcal{R}^{(p-r)\times(p-r)}$, and $Q_{4}\in\mathcal{R}^{(q-r)\times(q-r)}$ such that
\begin{align*}
\overline{R}=R\left[\begin{array}{ccc}
Q_{1} & 0 & 0\\
0 & Q_{2} & 0\\
0 & 0 & Q_{3}
\end{array}\right]\;\mbox{ and }\;
\overline{S}=S\left[\begin{array}{ccc}
Q_{1} & 0 & 0\\
0 & Q_{2} & 0\\
0 & 0 & Q_{4}
\end{array}\right].
\end{align*}
This allows us to conclude that
\begin{align*}
\overline{R}[Z_{1}(s)\ Z_{2}(s)]\overline{S}^{T}=R\left[\begin{array}{ccc}
Q_{1} & 0 & 0\\
0 & Q_{2} & 0\\
0 & 0 & Q_{3}
\end{array}\right][Z_{1}(s)\ Z_{2}(s)]\left[\begin{array}{ccc}
Q_{1}^{T} & 0 & 0\\
0 & Q_{2}^{T} & 0\\
0 & 0 & Q_{4}^{T}
\end{array}\right]S^{T}.
\end{align*}
Taking further nonsingular matrices $P_{1}\in\mathcal{R}^{m\times m}$, $P_{2}\in\mathcal{R}^{n\times n}$, and $B,B',C\in\mathcal{R}^{m\times n}$ with all the elements of $B$ being nonzero and with $B\circ B'=1_{m\times n}$, we can form
\begin{align*}
C'=B'\circ(P_{1}^{-1}(B\circ P_{1}CP_{2}))P_{2}^{-1}
\end{align*}
such that $B\circ P_{1}CP_{2}=P_{1}(B\circ C')P_{2}$. Therefore, there exists $D'$ satisfying 
\begin{align*}
\overline{R}[Z_{1}(s)\ Z_{2}(s)]\overline{S}^{T}=R[\widetilde{Z}_{1}(s)\ \widetilde{Z}_{2}(s)]S^{T},
\end{align*}
where $\widetilde{Z}_{1}(s)$ and $\widetilde{Z}_{2}(s)$ have the same expressions as $Z_{1}(s)$ and $Z_{2}(s)$, respectively, except that $D$ is replaced by $D'$. Then we fix $\overline{R}:=R$ and $\overline{S}:=S$. 
		 
Denote now $\widehat{D}_{1}^{s}:=\frac{\widehat{D}_{1}+\widehat{D}_{1}^{T}}{2}$, $\widehat{D}_{1}^{a}:=\frac{\widehat{D}_{1}-\widehat{D}_{1}^{T}}{2}$, $\widehat{D}_{1}:=R^{T}DS_{1}$, $\widehat{D}_{2}:=R^{T}DS_{2}$, $\widetilde{D}:=R^{T}DS=(d_{ij})_{p\times q}$, and by $\widehat{D}_{11}(s)$ the matrix extracted from the first $s$ rows and columns of $\widehat{D}_{1}$. For any $\mathcal{W}\in{\cal J}\operatorname{Prox}_{\|\cdot\|_{2}}(A)$, there are $\mathcal{V}^{k}\in\partial_{B}\operatorname{Prox}_{\|\cdot\|_{2}}(A)$, $\lambda_{k}\in[0,1]$, $k=1,\ldots,t$, with $\sum\limits_{k=1}^{t}\lambda_{k}=1$ such that $\mathcal{W}=\sum\limits_{k=1}^{t}\lambda_{k}\mathcal{V}^{k}$. For any $D\in\mathcal{R}^{p\times q}$, there are $s_{1},\ldots,s_{l}\in[k_{1}(\sigma),k_{2}(\sigma)]$, $s_{1}\leq\ldots\leq s_{t}$ satisfying the equality
\begin{align}\label{eq:partial-2-norm-expression}
\mathcal{W}D=D-R[E_{1}\ E_{2}]S^{T},
\end{align}
where $E_{1}\in\mathcal{R}^{p\times p}$ and $E_{2}\in\mathcal{R}^{p\times(q-p)}$ are defined by
\begin{align*}
E_{1}:=\widehat{\Omega}\circ\widehat{D}_{1}^{s}+\widehat{\Gamma}\circ\widehat{D}_{1}^{a}-\left[\begin{array}{cc}
\widehat{\Lambda} & 0 \\
0 & 0
\end{array}\right]\;\mbox{ and }\; E_{2}:=\left[\begin{array}{c}\widehat{\Upsilon}\,1_{q-p}^{T}\\ 0\end{array}\right]\circ\widehat{D}_{2}
\end{align*}
with $\widehat{\Omega}\in\mathcal{R}^{p\times p}$, $\widehat{\Gamma}\in\mathcal{R}^{p\times p}$, $\widehat{\Lambda}\in\mathcal{R}^{s_{t}\times s_{t}}$, and $\widehat{\Upsilon}\in\mathcal{R}^{s_{t}}$ given as
\begin{align*}
\widehat{\Omega}&:=\left[\begin{array}{cccc}
1_{k_{1}(\sigma)\times k_{1}(\sigma)} & \widehat{\Omega}_{\beta_{1}\gamma_{3}}& \widehat{\Omega}_{\beta_{1}\beta_{4}} & \widehat{\Omega}_{\beta_{1}\beta_{5}}\\
\widehat{\Omega}_{\beta_{1}\gamma_{3}}^{T} &  \\
\widehat{\Omega}_{\beta_{1}\beta_{4}}^{T} & & \widehat{\Omega}_{\gamma_{4}\gamma_{4}} \\\widehat{\Omega}_{\beta_{1}\beta_{5}}^{T} &
\end{array}\right],\ \begin{array}{ll}\gamma_{3}:=\{k_{1}(\sigma)+1,\ldots,k_{2}(\sigma)\},\\ \gamma_{4}:=\{k_{1}(\sigma)+1,\ldots,p\},\end{array}\\
\widehat{\Gamma}&:=\left[\begin{array}{cccc}
\widehat{\Gamma}_{\beta_{1}\beta_{1}} & \widehat{\Gamma}_{\beta_{1}\gamma_{3}} & \widehat{\Gamma}_{\beta_{1}\beta_{4}} & \widehat{\Gamma}_{\beta_{1}\beta_{5}} \\
\widehat{\Gamma}_{\beta_{1}\gamma_{3}}^{T} \\
\widehat{\Gamma}_{\beta_{1}\beta_{4}}^{T} & & \widehat{\Gamma}_{\gamma_{4}\gamma_{4}} \\
\widehat{\Gamma}_{\beta_{1}\beta_{5}}^{T}
\end{array}\right],\\
\widehat{\Lambda}&:=\operatorname{Diag}\left\{\underbrace{\sum_{l=1}^{t}\frac{\lambda_{l}\operatorname{Tr}D_{11}(s_{l})}{s_{l}},\dots,\sum_{l=1}^{t}\frac{\lambda_{l}\operatorname{Tr}D_{11}(s_{l})}{s_{l}}}_{k_{1}(\sigma)},\sum_{{l=1}\atop{s_{l}\geq k_{1}(\sigma)+1}}^{t}\frac{\lambda_{l}\operatorname{Tr}D_{11}(s_{l})}{s_{l}},\ldots,\sum_{{l=1}\atop{s_{l}\geq s_{t}}}^{t}\frac{\lambda_{l}\operatorname{Tr}D_{11}(s_{l})}{s_{l}}\right\}
\end{align*}
with their components defined via the previous constructions by
\begin{align*}
\widehat{\Omega}_{\beta_{1}\beta_{4}}&:= \Omega_{:\beta_{1}\beta_{4}},\ \widehat{\Omega}_{\beta_{1}\beta_{5}}:=\Omega_{\beta_{1}\beta_{5}},\ \widehat{\Gamma}_{\beta_{1}\beta_{1}}:=\Gamma_{\beta_{1}\beta_{1}},\ \widehat{\Gamma}_{\beta_{1}\beta_{4}}: = \Gamma_{\beta_{1}\beta_{4}},\ \widehat{\Gamma}_{\beta_{1}\beta_{5}}:=\Gamma_{\beta_{1}\beta_{5}},\\
\widehat{\Omega}_{\gamma_{4}\gamma_{4}}&:=\sum_{l=1}^{t}\lambda_{l}\Omega(s_{l})_{\gamma_{4}\gamma_{4}},\quad\widehat{\Omega}_{kj}:=\widehat{\Omega}_{jk}\in[0,1]\;
\mbox{ for all }\; k,j\in\gamma_{4},\\
(\widehat{\Omega}_{\beta_{1}\gamma_{3}})_{kj}&:=\sum_{{l=1}\atop{s_{l}\geq j+k_{1}(\sigma)}}^{t}\lambda_{l}+\sum_{{l=1}\atop{s_{l}<j+k_{1}(\sigma)}}^{t}\frac{\lambda_{l}p_{k}}{\sigma_{k}-\sigma_{j}}\;\mbox{ for all }\;k\in\beta_{1},\;j\in\gamma_{3}-k_{1}(\sigma),\\
\widehat{\Gamma}_{\gamma_{4}\gamma_{4}}&:=\sum_{l=1}^{t}\lambda_{l}\Gamma(s_{l})_{\gamma_{4}\gamma_{4}},\quad\widehat{\Gamma}_{kj}:=\widehat{\Gamma}_{jk}\in[0,1)\;
\mbox{ for all }\; k,j\in\gamma_{4},\\
(\widehat{\Gamma}_{\beta_{1}\gamma_{3}})_{kj}&:=\sum_{{l=1}\atop{s_{l}\geq j+k_{1}(\sigma)}}^{t}\frac{\lambda_{l}(p_{k}+p_{j})}{\sigma_{k}+\sigma_{j}}+\sum_{{l=1}\atop{s_{l}<j+k_{1}(\sigma)}}^{t}\frac{\lambda_{l}p_{k}}{\sigma_{k}+\sigma_{j}}\;\mbox{ for all }\;k\in\beta_{1},\;j\in\gamma_{3}-k_{1}(\sigma),\\
\widehat{\Upsilon}_{k}&:=\frac{p_{k}}{\sigma_{k}}\sum_{{l=1}\atop{s_{l}\geq k}}^{t}\lambda_{l}\;\mbox{ for all }\;k\in\{1,2,\ldots,s_{t}\}.
\end{align*}
For any $Y\in\bigcup\limits_{\mathcal{W}\in{\cal J}\operatorname{Prox}_{\|\cdot\|_{2}}(A)}\operatorname{rge}\mathcal{W}$, there exist $\mathcal{W}\in{\cal J}\operatorname{Prox}_{\|\cdot\|_{2}}(A)$ and $D\in\mathcal{R}^{p\times q}$ such that $Y=\mathcal{W}D$. Letting $\widetilde{Y}:=R^{T}YS=(y_{ij})_{p\times q}$, it follows from \eqref{eq:partial-2-norm-expression} that
\begin{align*}
&\left\langle Y,D-Y\right\rangle=\left\langle\widetilde{Y},\widetilde{D}-\widetilde{Y}\right\rangle=\sum_{{1\leq i,j\leq k_{1}(\sigma)}\atop{i\neq j}}y_{ij}(d_{ij}-y_{ij})+\left\langle \widetilde{Y}_{\beta_{1}\gamma_{4}},(\widetilde{D}-\widetilde{Y})_{\beta_{1}\gamma_{4}}\right\rangle \\
&\ +\left\langle \widetilde{Y}_{\gamma_{4}\beta_{1}},(\widetilde{D}-\widetilde{Y})_{\gamma_{4}\beta_{1}}\right\rangle+\left\{\left\langle \widetilde{Y}_{\gamma_{4}\gamma_{4}},(\widetilde{D}-\widetilde{Y})_{\gamma_{4}\gamma_{4}}\right\rangle+\sum_{i=1}^{k_{1}(\sigma)}y_{ii}(d_{ii}-y_{ii})\right\}+\sum_{{1\leq i\leq k_{1}(\sigma)}\atop{p+1\leq j\leq q}}y_{ij}(d_{ij}-y_{ij})\\
&\geq2\sum_{1\leq i<j\leq k_{1}(\sigma)}\frac{p_{i}+p_{j}}{\sigma_{i}+\sigma_{j}}\left(1-\frac{p_{i}+p_{j}}{\sigma_{i}+\sigma_{j}}\right)\left(\frac{d_{ij}-d_{ji}}{2}\right)^{2}\\
&\quad+2\sum_{{1\leq i\leq k_{1}(\sigma)}\atop{k_{1}(\sigma)+1\leq j\leq k_{2}(\sigma)}}\left[\widehat{\Omega}_{ij}\left(1-\widehat{\Omega}_{ij}\right)\left(\frac{d_{ij}+d_{ji}}{2}\right)^{2}+\widehat{\Gamma}_{ij}\left(1-\widehat{\Gamma}_{ij}\right)\left(\frac{d_{ij}-d_{ji}}{2}\right)^{2}\right]\\
&\quad+2\sum_{{1\leq i\leq k_{1}(\sigma)}\atop{k_{2}(\sigma)+1\leq j\leq r}}\left[\frac{p_{i}}{\sigma_{i}-\sigma_{j}}\left(1-\frac{p_{i}}{\sigma_{i}-\sigma_{j}}\right)\left(\frac{d_{ij}+d_{ji}}{2}\right)^{2}+\frac{p_{i}}{\sigma_{i}+\sigma_{j}}\left(1-\frac{p_{i}}{\sigma_{i}+\sigma_{j}}\right)\left(\frac{d_{ij}-d_{ji}}{2}\right)^{2}\right]\\
&\quad+2\sum_{{1\leq i\leq k_{1}(\sigma)}\atop{r+1\leq j\leq p}}\frac{p_{i}}{\sigma_{i}}\left(1-\frac{p_{i}}{\sigma_{i}}\right)\left[\left(\frac{d_{ij}+d_{ji}}{2}\right)^{2}+\left(\frac{d_{ij}-d_{ji}}{2}\right)^{2}\right]+\sum_{i=1}^{k_{1}(\sigma)}\sum_{j=p+1}^{q}\frac{p_{i}}{\sigma_{i}}\left(1-\frac{p_{i}}{\sigma_{i}}\right)d_{ij}^{2}
\end{align*}
\begin{align*}
&=2\sum_{1\leq i<j\leq k_{1}(\sigma)}\frac{p_{i}+p_{j}}{\sigma_{i}+\sigma_{j}-p_{i}-p_{j}}\left(\frac{y_{ij}-y_{ji}}{2}\right)^{2}+2\sum_{{1\leq i\leq k_{1}(\sigma)}\atop{k_{1}(\sigma)+1\leq j\leq k_{2}(\sigma)}}\left[\frac{\widehat{\Omega}_{ij}}{1-\widehat{\Omega}_{ij}}\left(\frac{y_{ij}+y_{ji}}{2}\right)^{2}\right.\\
&\quad\left.+\frac{\widehat{\Gamma}_{ij}}{1-\widehat{\Gamma}_{ij}}\left(\frac{y_{ij}-y_{ji}}{2}\right)^{2}\right]+2\sum_{{1\leq i\leq k_{1}(\sigma)}\atop{k_{2}(\sigma)+1\leq j\leq r}}\left[\frac{p_{i}}{\sigma_{i}-\sigma_{j}-p_{i}}\left(\frac{y_{ij}+y_{ji}}{2}\right)^{2}+\frac{p_{i}}{\sigma_{i}+\sigma_{j}-p_{i}}\left(\frac{y_{ij}-y_{ji}}{2}\right)^{2}\right]\\
&\quad+2\sum_{{1\leq i\leq k_{1}(\sigma)}\atop{r+1\leq j\leq p}}\frac{p_{i}}{\sigma_{i}-p_{i}}\left[\left(\frac{y_{ij}+y_{ji}}{2}\right)^{2}+\left(\frac{y_{ij}-y_{ji}}{2}\right)^{2}\right]+\sum_{i=1}^{k_{1}(\sigma)}\sum_{j=p+1}^{q}\frac{p_{i}}{\sigma_{i}-p_{i}}y_{ij}^{2}\\
&\geq2\sum_{1\leq i<j\leq k_{1}(\sigma)}\frac{p_{i}+p_{j}}{\sigma_{i}+\sigma_{j}-p_{i}-p_{j}}\left(\frac{y_{ij}-y_{ji}}{2}\right)^{2}\\
&\quad+2\sum_{{1\leq i\leq k_{1}(\sigma)}\atop{k_{1}(\sigma)+1\leq j\leq r}}\left[\frac{p_{i}}{\sigma_{i}-\sigma_{j}-p_{i}}\left(\frac{y_{ij}+y_{ji}}{2}\right)^{2}+\frac{p_{i}}{\sigma_{i}+\sigma_{j}-p_{i}}\left(\frac{y_{ij}-y_{ji}}{2}\right)^{2}\right]\\
&\quad+2\sum_{{1\leq i\leq k_{1}(\sigma)}\atop{r+1\leq j\leq p}}\frac{p_{i}}{\sigma_{i}-p_{i}}\left[\left(\frac{y_{ij}+y_{ji}}{2}\right)^{2}+\left(\frac{y_{ij}-y_{ji}}{2}\right)^{2}\right]+\sum_{i=1}^{k_{1}(\sigma)}\sum_{j=p+1}^{q}\frac{p_{i}}{\sigma_{i}-p_{i}}y_{ij}^{2},
\end{align*}
where we have $y_{ij}+y_{ji}=0$ if $\widehat{\Omega}_{ij}=1$ with the convention that  
$\frac{\widehat{\Omega}_{ij}}{1-\widehat{\Omega}_{ij}}\left(\frac{y_{ij}+y_{ji}}{2}\right)^{2}:=0$, and where the first inequality is derived from the nonnegativity of the term $\{\cdots\}$, which is guaranteed by \cite[Proposition~3.2]{Tangwang2025} for the specific $Y$ satisfying $y_{ij}=y_{ji}=0$ whenever $1\leq i<j\leq k_{1}(\sigma)$, or, $i\in\beta_{1}$ and $j\in\gamma_{4}$, the last inequality holds as an equality with $\widehat{\Omega}_{ij}=\frac{p_{i}}{\sigma_{i}-\sigma_{j}}$ and $\widehat{\Gamma}_{ij}=\frac{p_{i}}{\sigma_{i}+\sigma_{j}}$ for $1\leq i\leq k_{1}(\sigma)$, $k_{1}(\sigma)+1\leq j\leq k_{2}(\sigma)$, and $s_{l}=k_{1}(\sigma)$ for $l=1,\ldots,t$. Setting $s:=k_{1}(\sigma)$ yields  $\beta_{2}=\emptyset$. For this specific $\overline{\mathcal{W}}\in\partial_{B}\operatorname{Prox}_{\|\cdot\|_{2}}(A)$, we conclude that
\begin{align*}
&\bigcup_{\mathcal{W}\in{\cal J}\operatorname{Prox}_{\|\cdot\|_{2}}(A)}\operatorname{rge}\mathcal{W}=\operatorname{rge}\overline{\mathcal{W}}\\
&=\left\{D-R[Z_{1}(s)\ Z_{2}(s)]S^{T}\, \left|\,Z_{1}(s)=\Omega(s)\circ \widehat{D}_{1}^{s}+\Gamma(s)\circ \widehat{D}_{1}^{a}-\frac{\operatorname{Tr}(\widehat{D}_{11}(s))}{s}\left[\begin{array}{cc}
I_{s} & 0\\ 0 & 0
\end{array}\right],\right.\right.\\
&\quad\quad\left.Z_{2}(s)=\left[\begin{array}{c}
\Upsilon\, 1^{T}_{q-p} \\0
\end{array}\right]\circ \widehat{D}_{2},\,s=k_{1}(\sigma),\, D\in\mathcal{R}^{p\times q}\right\}.
\end{align*}
Moreover, for any $Y\in\bigcup\limits_{\mathcal{W}\in{\cal J}\operatorname{Prox}_{\|\cdot\|_{2}}(A)}\operatorname{rge}\mathcal{W}$ with $R^{T}YS:=(y_{ij})_{p\times q}$, there exists $\overline{D}\in\mathcal{R}^{p\times q}$ with $R^{T}\overline{D}S:=(\bar{d}_{ij})_{p\times q}$ satisfying the conditions

\begin{align*}
\frac{1}{k_{1}(\sigma)}\sum_{i=1}^{k_{1}(\sigma)}\bar{d}_{ii}&=y_{ii},\,\forall\,i\in\beta_{1},\\
\frac{\bar{d}_{ij}-\bar{d}_{ji}}{2}&=\frac{\sigma_{i}+\sigma_{j}}{\sigma_{i}+\sigma_{j}-p_{i}-p_{j}}y_{ij}=-\frac{\sigma_{i}+\sigma_{j}}{\sigma_{i}+\sigma_{j}-p_{i}-p_{j}}y_{ji},\quad\forall\,i<j,\,i,j\in\beta_{1},\\
\bar{d}_{ij}&=\frac{\sigma_{i}-\sigma_{j}}{\sigma_{i}-\sigma_{j}-p_{i}}\frac{y_{ij}+y_{ji}}{2}+\frac{\sigma_{i}+\sigma_{j}}{\sigma_{i}+\sigma_{j}-p_{i}}\frac{y_{ij}-y_{ji}}{2},\quad \forall\,i\in\beta_{1},\,j\in\gamma_{3}\cup\beta_{4},\\
\bar{d}_{ji}&=\frac{\sigma_{i}-\sigma_{j}}{\sigma_{i}-\sigma_{j}-p_{i}}\frac{y_{ij}+y_{ji}}{2}-\frac{\sigma_{i}+\sigma_{j}}{\sigma_{i}+\sigma_{j}-p_{i}}\frac{y_{ij}-y_{ji}}{2},\quad \forall\,i\in\beta_{1},\,j\in\gamma_{3}\cup\beta_{4},\\
\bar{d}_{ij}&=\frac{\sigma_{i}}{\sigma_{i}-p_{i}}y_{ij},\quad
\bar{d}_{ji}=\frac{\sigma_{i}}{\sigma_{i}-p_{i}}y_{ji},\quad\forall\,i\in\beta_{1},\,j\in\beta_{5},\\
\bar{d}_{ij}&=\frac{\sigma_{i}}{\sigma_{i}-p_{i}}y_{ij},\quad \forall\,i\in\beta_{1},\,j=p+1,\ldots,q,\\
\bar{d}_{ij}&=y_{ij},\quad\forall\, i=k_{1}(\sigma)+1,\ldots,p,\,j=k_{1}(\sigma)+1,\ldots,q,
\end{align*} 
and being such that $Y=\overline{\mathcal{W}}\overline{D}$ with the equalities
\begin{align*}
\left\langle Y,\overline{D}-Y\right\rangle&=\min_{{D\in\mathcal{R}^{p\times q},Y=\mathcal{W}D}\atop{\mathcal{W}\in{\cal J}\operatorname{Prox}_{\|\cdot\|_{2}}(A)}}\left\langle Y,D-Y\right\rangle\\
&=2\sum_{1\leq i<j\leq k_{1}(\sigma)}\frac{p_{i}+p_{j}}{\sigma_{i}+\sigma_{j}-p_{i}-p_{j}}\left(\frac{y_{ij}-y_{ji}}{2}\right)^{2}\\
&\quad+2\sum_{{1\leq i\leq k_{1}(\sigma)}\atop{k_{1}(\sigma)+1\leq j\leq r}}\left[\frac{p_{i}}{\sigma_{i}-\sigma_{j}-p_{i}}\left(\frac{y_{ij}+y_{ji}}{2}\right)^{2}+\frac{p_{i}}{\sigma_{i}+\sigma_{j}-p_{i}}\left(\frac{y_{ij}-y_{ji}}{2}\right)^{2}\right]\\
&\quad+2\sum_{{1\leq i\leq k_{1}(\sigma)}\atop{r+1\leq j\leq p}}\frac{p_{i}}{\sigma_{i}-p_{i}}\left[\left(\frac{y_{ij}+y_{ji}}{2}\right)^{2}+\left(\frac{y_{ij}-y_{ji}}{2}\right)^{2}\right]+\sum_{i=1}^{k_{1}(\sigma)}\sum_{j=p+1}^{q}\frac{p_{i}}{\sigma_{i}-p_{i}}y_{ij}^{2}.
\end{align*}
This confirms that Assumption~\ref{assump-2} holds for the function $g=\|\cdot\|_{2}$ therein.
\end{appendices}\vspace*{-0.12in}

\section*{Statements and Declarations}
Research of Boris Mordukhovich was partly supported by the US National Science Foundation under grant DMS-2204519, by the Australian Research Council under Discovery Project DP-190100555, and by Project 111 of China under grant D21024.
All the authors contributed equally and do not have any competing interests, concent to participate and publish. All the data are available, and the ethics approval is not required.\vspace*{-0.12in}

\bibliographystyle{amsplain}
\bibliography{ref.bib}
\end{document}